\documentclass[12pt,reqno]{amsart}

\usepackage[usenames, dvipsnames, svgnames]{xcolor}
\usepackage{amsmath, amssymb,graphicx,amsthm,latexsym, amsfonts, enumitem, mathtools, tensor, MnSymbol}
\usepackage{hyperref}
\usepackage[all, color]{xy}
\usepackage{color}
\usepackage{amssymb}
\usepackage{float}
\usepackage{tikz}
\usetikzlibrary{arrows,decorations.pathmorphing,decorations.pathreplacing,positioning,shapes.geometric,shapes.misc,decorations.markings,decorations.fractals,calc,patterns}

\DeclareMathOperator{\Hom}{Hom}
\DeclareMathOperator{\End}{End}

\DeclareMathOperator{\im}{Im}
\DeclareMathOperator{\Ker}{Ker}
\DeclareMathOperator{\Coker}{Coker}
\DeclareMathOperator{\Mod}{Mod}
\DeclareMathOperator{\mmod}{mod}

\DeclareMathOperator{\Ext}{Ext}

\theoremstyle{plain}
\newtheorem{theorem}{Theorem}[section]
\newtheorem*{theorem*}{Theorem}

\theoremstyle{definition}
\newtheorem{defn}[theorem]{Definition}

\newtheorem{exmp}[theorem]{Example} 
\newtheorem{remark}[theorem]{Remark}

\newtheorem{lemma}[theorem]{Lemma}
\newtheorem{notation}[theorem]{Notation}

\newtheorem{setup}[theorem]{Setup}
\newtheorem{proposition}[theorem]{Proposition}

\date{}
\begin{document}
\setlength{\parindent}{0pt}
\setlength{\parskip}{7pt}
\title[Grothendieck groups of triangulated categories]{Grothendieck groups of triangulated categories via cluster tilting subcategories}
\author{Francesca Fedele}
\address{School of Mathematics, Statistics and Physics,
Newcastle University, Newcastle upon Tyne NE1 7RU, United Kingdom}
\email{F.Fedele2@newcastle.ac.uk}
\maketitle
\begin{abstract}
    Let $k$ be a field and $\mathcal{C}$ a $k$-linear, Hom-finite triangulated category with split idempotents. In this paper, we show that under suitable circumstances, the Grothendieck group of $\mathcal{C}$, denoted $K_0(\mathcal{C})$, can be expressed as a quotient of the split Grothendieck group of a higher cluster tilting subcategory of $\mathcal{C}$. The results we  prove are higher versions of results on Grothendieck groups of triangulated categories by Xiao and Zhu and by Palu.
   Assume that $n\geq 2$ is an integer, $\mathcal{C}$ has a Serre functor $\mathbb{S}$ and an $n$-cluster tilting subcategory $\mathcal{T}$ such that Ind$\,\mathcal{T}$ is locally bounded. Then, for every indecomposable $M$ in $\mathcal{T}$, there is an Auslander--Reiten $(n+2)$-angle in $\mathcal{T}$ of the form $\mathbb{S}\Sigma^{-n}(M)\rightarrow T_{n-1}\rightarrow\dots\rightarrow T_0\rightarrow M$ and
    \begin{align*}
        K_0(\mathcal{C})\cong K_0^{sp}(\mathcal{T})\Big/\Big\langle -[M]+(-1)^n[\mathbb{S}\Sigma^{-n}(M)]+\sum_{i=0}^{n-1}(-1)^i[T_i]\Big\mid M\in\text{Ind }\mathcal{T} \Big\rangle.
    \end{align*}
    Assume now that $d$ is a positive integer and $\mathcal{C}$ has a $d$-cluster tilting subcategory $\mathcal{S}$ closed under $d$-suspension. Then $\mathcal{S}$ is a so called $(d+2)$-angulated category whose Grothendieck group $K_0(\mathcal{S})$ can be defined as a certain quotient of $K_0^{sp}(\mathcal{S})$. We will show
    \begin{align*}
        K_0(\mathcal{C})\cong K_0(\mathcal{S}).
    \end{align*}
Moreover, assume that $n=2d$, that all the above assumptions hold, and that $\mathcal{T}\subseteq \mathcal{S}$. Then our results can be combined to express $K_0(\mathcal{S})$ as a quotient of $K_0^{sp}(\mathcal{T})$.
\end{abstract}

\section{Introduction}

Let $k$ be a field and $\mathcal{C}$ be a $k$-linear, $\Hom$-finite triangulated category with split idempotents and suspension functor $\Sigma$.  We denote the split Grothendieck group of an additive category $\mathcal{A}$ by $K_0^{sp}(\mathcal{A})$ and the Grothendieck group of an abelian or triangulated category $\mathcal{B}$ by $K_0(\mathcal{B})$.

We first present two previous results, one by Xiao and Zhu and the other one by Palu that in some sense are the base case of our results. Note that both Xiao and Zhu and Palu assume that $k$ is an algebraically closed field. However, this assumption is not needed for our higher versions and $k$ is a general field in our setup.

{\bf Previous results.}
Xiao and Zhu presented triangulated analogues of results of Auslander \cite[Theorems 2.2 and 2.3]{AM} and Butler \cite[Theorem in introduction]{BMCR}. In particulat, they proved in \cite[Theorem 2.1]{XZ}, that if $\mathcal{C}$ is a triangulated category of finite type, then
\begin{align*}
    K_0(\mathcal{C})\cong K_0^{sp}(\mathcal{C})/\langle -[A]+[B]-[C]\mid C\in\rm{Ind }\,\mathcal{C} \text{ with Auslander--Reiten triangle } \it{A\rightarrow B\rightarrow C} \rangle.
\end{align*}

We can think of the $\mathcal{C}$ appearing on the right side as the only possible $1$-cluster tilting subcategory of $\mathcal{C}$. In this paper, we are interested in higher cluster tilting subcategories.

The first higher case occurs when $\mathcal{C}$ has a ($2$-)cluster tilting subcategory. Palu studied this case, in a more specific setup, in \cite{PY}. In fact, Palu assumes that $\mathcal{C}$ is the stable category of a Frobenius $k$-linear category with split idempotents, and that $\mathcal{C}$ is $2$-Calabi--Yau with a ($2$-)cluster tilting subcategory $\mathcal{T}$.

Given an indecomposable $M$ in $\mathcal{T}$, let $\widetilde{\mathcal{T}}$ be the additive subcategory of $\mathcal{T}$ whose indecomposables are the same as $\mathcal{T}$, excluding those isomorphic to $M$.
Then, up to isomorphism, there is a unique indecomposable $M^*\not\in\mathcal{T}$ such that $\text{add}(\widetilde{\mathcal{T}}\cup M^*)\subseteq \mathcal{C}$ is ($2$-)cluster tilting. Moreover, $M$ and $M^*$ appear in two triangles with certain properties, called \textit{exchange triangles}, of the form:
\begin{align*}
    M^*\rightarrow B_M\rightarrow M\rightarrow \Sigma M^* \text{ and }
    M\rightarrow B_{M^*}\rightarrow M^*\rightarrow \Sigma M, 
\end{align*}
where $B_M$ and $B_{M^*}$ are in $\widetilde{\mathcal{T}}$. Palu proved in \cite[Theorem 10]{PY} that 
\begin{align*}
K_0(\mathcal{C})\cong K_0^{sp}(\mathcal{T})/\langle [B_{M^*}]-[B_M] \rangle_{M}.
\end{align*}

Note that if the Auslander--Reiten quiver of $\mathcal{T}$ has no loops, then the indecomposable $M\in \mathcal{T}$ has Auslander--Reiten $4$-angle $M\rightarrow B_{M^*}\rightarrow B_M\rightarrow M$ in the sense of \cite[Theorem 3.8]{IY}. So Palu's theorem is a higher version of Xiao and Zhu's theorem.

We present ``higher cluster tilting'' versions of Xiao and Zhu and Palu's results. Moreover, we present a ``higher angulated'' version of Palu's result.

{\bf Higher cluster tilting versions.} Let $n\geq 2$ be an integer and assume that $\mathcal{C}$ has a Serre functor $\mathbb{S}$ and an $n$-cluster tilting subcategory $\mathcal{T}$, see Definition \ref{defn_m-ct}. Let Ind$\,\mathcal{T}$ be a full subcategory of $\mathcal{T}$ containing precisely one object from each isomorphism class of indecomposables in $\mathcal{T}$ and assume that Ind$\,\mathcal{T}$ is locally bounded, see Definition \ref{defn_lb}. Recall that the functor $\mathbb{S}_n:=\mathbb{S}\Sigma^{-n}$ and Auslander--Reiten $(n+2)$-angles in $\mathcal{T}$ were introduced in \cite[Section 3]{IY}.

{\bf Theorem A.}
{\em We have that $K_0(\mathcal{C})$ is isomorphic to 
\begin{align*}
K_0^{sp}(\mathcal{T})\Bigg/\Bigg\langle \xymatrix{ {\begin{matrix} -[M]+(-1)^n[\mathbb{S}_n(M)]+\\\sum_{i=0}^{n-1} (-1)^i [T_i]\end{matrix}}\,\,\,\Bigg\vert\,\,\, {\begin{matrix}M\in\rm{Ind }\,\mathcal{T} \text{ with Auslander--Reiten } \text{$(n+2)$-angle }\\
    \mathbb{S}_n(M)\rightarrow T_{n-1}\rightarrow\dots\rightarrow T_0\rightarrow M\rightarrow \mathbb{S}(M)\end{matrix}}
   }\Bigg\rangle.
\end{align*}
}

The arguments we use to prove Theorem A rely on $n\geq 2$. However, note that when $k$ is an algebraically closed field, the case $n=1$ is still true and it is an instance of Xiao and Zhu's theorem.

If we add certain extra assumptions, we obtain the following because $\mathbb{S}_n\cong id$.

{\bf Corollary B.}
{\em Assume that $n\geq 2$ is an even integer and $\mathcal{C}$ is $n$-Calabi--Yau. Then
\begin{align*}
K_0(\mathcal{C})\cong K_0^{sp}(\mathcal{T})\Bigg/\Bigg\langle \xymatrix{ \displaystyle\sum_{i=0}^{n-1} (-1)^i [T_i]\,\,\,\Bigg\vert\,\,\, {\begin{matrix}M\in\rm{Ind  }\,\mathcal{T} \text{ with Auslander--Reiten } \text{$(n+2)$-angle }\\
    M\rightarrow T_{n-1}\rightarrow\dots\rightarrow T_0\rightarrow M\rightarrow \Sigma^{n}M\end{matrix}}
   }\Bigg\rangle.
\end{align*}
}

When $n=2$ and the Auslander--Reiten quiver of $\mathcal{T}$ has no loops, then Corollary B and Palu's theorem coincide.

{\bf Higher angulated version.} Let $d\geq 1$ be an integer and assume that $\mathcal{C}$ has a $d$-cluster tilting subcategory $\mathcal{S}$ such that $\Sigma^d\mathcal{S}=\mathcal{S}$. Note that $\mathcal{S}$ is a $(d+2)$-angulated category with $d$-suspension $\Sigma^d$, by \cite[Theorem 1]{GKO}.
Similarly to the way $K_0(\mathcal{C})$ is defined, one can define the Grothendieck group of the $(d+2)$-angulated category $\mathcal{S}$ as
\begin{align*}
    K_0(\mathcal{S}):=K_0^{sp}(\mathcal{S})\Bigg/ \Bigg\langle \sum_{i=0}^{d+1} (-1)^i [S_i]\,\,\Bigg\vert\,\, S_{d+1}\rightarrow\dots\rightarrow S_0\rightarrow\Sigma^d S_{d+1} \text{ is a $(d+2)$-angle in } \mathcal{S} \Bigg\rangle,
\end{align*}
see \cite[Definition 2.1]{BT}.
We prove that this is isomorphic to the Grothendieck group of $\mathcal{C}$.

{\bf Theorem C.}
{\em We have that $K_0(\mathcal{C})\cong K_0(\mathcal{S})$.
}

Let $n=2d$. We now add the assumptions that $\mathcal{C}$ has a Serre functor $\mathbb{S}$ and that there is an $n$-cluster tilting subcategory $\mathcal{T}\subseteq\mathcal{C}$
such that $\mathcal{T}\subseteq\mathcal{S}$ and Ind$\,\mathcal{T}$ is locally bounded. By \cite[Theorem 5.26]{OT}, we have that $\mathcal{T}\subseteq\mathcal{S}$ is an Oppermann--Thomas cluster tilting subcategory, \textbf{i.e.} the corresponding concept in a $(d+2)$-angulated category of a cluster tilting subcategory in a triangulated category.
Theorems A and C have the following immediate consequence.

{\bf Theorem D.}
{\em
We have that $K_0(\mathcal{C})\cong K_0(\mathcal{S})$ and
\begin{align*}
    K_0(\mathcal{S})\cong K_0^{sp}(\mathcal{T})\Bigg/\Bigg\langle \xymatrix{ {\begin{matrix} -[M]+[\mathbb{S}_n(M)]+\\\sum_{i=0}^{n-1} (-1)^i [T_i]\end{matrix}}\,\,\,\Bigg\vert\,\,\, {\begin{matrix}M\in\rm{Ind }\,\mathcal{T} \text{ with Auslander--Reiten } \text{$(n+2)$-angle }\\
    \mathbb{S}_n(M)\rightarrow T_{n-1}\rightarrow\dots\rightarrow T_0\rightarrow M\rightarrow \mathbb{S}(M)\end{matrix}}
   }\Bigg\rangle.
\end{align*}
}


{\bf Corollary E.}
If $\mathcal{C}$ is $n$-Calabi--Yau, then
\begin{align*}
    K_0(\mathcal{S})\cong K_0^{sp}(\mathcal{T})\Bigg/\Bigg\langle \xymatrix{ \displaystyle\sum_{i=0}^{n-1} (-1)^i [T_i]\,\,\,\Bigg\vert\,\,\, {\begin{matrix}M\in\rm{Ind }\,\mathcal{T} \text{ with Auslander--Reiten } \text{$(n+2)$-angle }\\
    M\rightarrow T_{n-1}\rightarrow\dots\rightarrow T_0\rightarrow M\rightarrow \Sigma^{n}M\end{matrix}}
   }\Bigg\rangle.
\end{align*}

When $d=1$, we have that $\mathcal{S}=\mathcal{C}$ is a triangulated category with ($2$-)cluster tilting subcategory $\mathcal{T}$ and, adding the extra assumption that the Auslander--Reiten quiver of $\mathcal{T}$ has no loops, Corollary E becomes Palu's theorem. For higher values of $d$, Corollary E proves a higher angulated version of Palu's theorem.

We conclude our paper by illustrating our results in two examples: one for each of the higher versions of Palu's theorem.
Let $q$ and $p$ be integers and $q$ be odd.
Consider the triangulated $q$-cluster category of Dynkin type $A_p$, denoted $\mathcal{C}_q(A_p)$, and note that this is a $(q+1)$-Calabi--Yau category. Since $q+1$ is even by assumption, we can apply Corollary B to show that
\begin{align*}
    K_0(\mathcal{C}_q(A_p))\cong
    \begin{cases}
    0, &\text{if $p$ is even,}\\
    \mathbb{Z}, & \text{if $p$ is odd.}
    \end{cases}
\end{align*}

We then consider a higher Auslander algebra $A^2_3$ of Dynkin type $A$ and its Amiot cluster category $\mathcal{C}^{4}(A^2_3)$, to find an example of categories
\begin{align*}
    \mathcal{T}\subseteq \mathcal{O}({A_3^2})\subseteq \mathcal{C}^{4}(A^2_3),
\end{align*}
such that $\mathcal{C}^{4}(A^2_3)$ is triangulated and $4$-Calabi--Yau, $\mathcal{O}({A_3^2})$ is closed under $\Sigma^2$ and $2$-cluster tilting in $\mathcal{C}^{4}(A^2_3)$ and $\mathcal{T}$ is $4$-cluster tilting in $\mathcal{C}^{4}(A^2_3)$. Applying Theorem C and Corollary E to this example, we find that
\begin{align*}
    K_0 (\mathcal{C}^{4}(A^2_3))\cong \mathbb{Z}\oplus\mathbb{Z}.
\end{align*}

The paper is organised as follows. Section \ref{section_setup} recalls some definitions and results and presents our setup. Section \ref{section_Gr_maps} introduces some morphisms between Grothendieck groups that will be useful in the rest of the paper. Section \ref{section_thm_T} proves Theorem A. Section \ref{section_thm_S} proves Theorem C. Section \ref{section_coro} presents Corollary E. Finally, Sections \ref{section_eg1} and \ref{section_eg2} illustrate our two examples.
\section{Setup and definitions}\label{section_setup}

\begin{defn}
Let $\mathcal{A}$ be an essentially small additive category and $G(\mathcal{A})$ be the free abelian group on isomorphism classes $[A]$ of objects $A\in\mathcal{A}$. We define the \textit{split Grothendieck group of $\mathcal{A}$} to be
\begin{align*}
    K_0^{sp}(\mathcal{A}):= G(\mathcal{A})/\langle [A\oplus B]-[A]-[B] \rangle.
\end{align*}
When $\mathcal{A}$ is abelian or triangulated, we can also define the \textit{Grothendieck group of $\mathcal{A}$} as
\begin{align*}
    K_0(\mathcal{A}):=& K_0^{sp}(\mathcal{A})/\langle [A]-[B]+[C]\mid 0\rightarrow A\rightarrow B\rightarrow C\rightarrow 0 \text{ is a short exact sequence in } \mathcal{A} \rangle \text{ or}\\
    K_0(\mathcal{A}):=& K_0^{sp}(\mathcal{A})/\langle [A]-[B]+[C]\mid A\rightarrow B\rightarrow C\rightarrow \Sigma A \text{ is a triangle in } \mathcal{A} \rangle,
\end{align*}
 respectively.
\end{defn}
In a similar way, one can define the Grothendieck group of a $(d+2)$-angulated category $\mathcal{S}$ as follows. 
\begin{defn}\label{defn_GG_S}
Let $d$ be a positive integer. The \textit{Grothendieck group of a $(d+2)$-angulated category $\mathcal{S}$} with $d$-suspension functor $\Sigma^d$ is defined to be
\begin{align*}
    K_0(\mathcal{S}):=K_0^{sp}(\mathcal{S})\Bigg/ \Bigg\langle \sum_{i=0}^{d+1} (-1)^i [S_i]\mid S_{d+1}\rightarrow\dots\rightarrow S_0\rightarrow\Sigma^d S_{d+1} \text{ is a $(d+2)$-angle in } \mathcal{S} \Bigg\rangle.
\end{align*}
\end{defn}
\begin{remark}
The definition of the Grothendieck group of a $(d+2)$-angulated category $\mathcal{S}$ we just introduced is different from the original one, see \cite[Definition 2.1]{BT}, however the two definitions are equivalent. In fact, since we define $K_0(\mathcal{S})$ as a quotient of $K_0^{sp}(\mathcal{S})$, the relation $[0]=0$ holds. On the other hand, in \cite{BT}, $K_0(\mathcal{S})$ is defined as a quotient of the free abelian group on the set of isomorphism classes of objects in $\mathcal{S}$ and the relation $[0]=0$ has to be manually added when $d$ is even.
\end{remark}

\begin{defn}
Let $\mathcal{U}\subseteq\mathcal{C}$ be a full subcategory. A $\mathcal{U}$\textit{-precover} of $C\in\mathcal{C}$ is a morphism of the form $\mu :U\rightarrow C$ with $U\in \mathcal{U}$ such that every morphism $\mu':U'\rightarrow C$ with $U'\in \mathcal{U}$ factorizes as:
\begin{align*}
\xymatrix{
U'\ar[rr]^{\mu'} \ar@{-->}[dr]_{\exists}& & C.\\
& U \ar[ru]_{\mu} &
}
\end{align*}
A $\mathcal{U}$\textit{-cover} of $C$ is a $\mathcal{U}$-precover $\mu: U\rightarrow C$ of $C$ which is also a right minimal morphism, that is if $\varphi: U\rightarrow U$ is a morphism such that $\mu\circ\varphi=\mu$, then $\varphi$ is an isomorphism.
The dual notions of precovers and covers are \textit{preenvelopes} and \textit{envelopes} respectively.

The subcategory $\mathcal{U}$ of $\mathcal{C}$ is called \textit{precovering} if every object in $\mathcal{C}$ has a $\mathcal{U}$-precover. Dually, it is called \textit{preenveloping} if every object in $\mathcal{C}$ has a $\mathcal{U}$-preenvelope. If $\mathcal{U}$ is both precovering and preenveloping, it is called \textit{functorially finite} in $\mathcal{C}$.
\end{defn}
\begin{defn}[{\cite[Section 3]{IY}}]\label{defn_m-ct}
For $m\geq 1$ an integer, an \textit{$m$-cluster tilting subcategory of $\mathcal{C}$} is a functorially finite, full subcategory $\mathcal{U}$ of $\mathcal{C}$ satisfying
\begin{align*}
    \mathcal{U}=\{ C\in\mathcal{C}\mid \Ext^{1\dots m-1}_{\mathcal{C}}(\mathcal{U},C)=0 \}=\{ C\in\mathcal{C}\mid \Ext^{1\dots m-1}_{\mathcal{C}}(C,\mathcal{U})=0 \}.
\end{align*}
\end{defn}
\begin{remark}
Note that when $m=1$ in Definition \ref{defn_m-ct}, the conditions $\Ext^{1\dots m-1}_{\mathcal{C}}(\mathcal{U},C)=0$ and $\Ext^{1\dots m-1}_{\mathcal{C}}(C,\mathcal{U})=0$ are empty and $\mathcal{C}=\mathcal{U}$ is the only possible $1$-cluster tilting subcategory.
\end{remark}

\begin{setup}
Let $m\geq 2$ be an integer and $\mathcal{U}$ an $m$-cluster tilting subcategory of $\mathcal{C}$.
\end{setup}

\begin{defn}[{\cite[Definition 2.9]{IY}}]
A \textit{$\mathcal{U}$-module} is a contravariant $k$-linear functor of the form $G:\mathcal{U}\rightarrow\Mod k$. Then $\mathcal{U}$-modules form an abelian category denoted  $\Mod\mathcal{U}$. We say that $G\in\Mod\mathcal{U}$ is \textit{coherent} if there exists an exact sequence of the form
\begin{align*}
    \mathcal{U}(-,U_1)\rightarrow\mathcal{U}(-,U_0)\rightarrow G(-)\rightarrow 0, 
\end{align*}
for some $U_1,\,U_0\in\mathcal{U}$. We denote by $\mmod\mathcal{U}$ the full subcategory of $\Mod\mathcal{U}$ consisting of coherent $\mathcal{U}$-modules.
\end{defn}

\begin{defn}
There is a homological functor
\begin{align*}
    F_{\mathcal{U}}:\mathcal{C}\rightarrow\mmod\mathcal{U},\,\,\, C\mapsto\mathcal{C}(-,C)|_\mathcal{U}.
\end{align*}
\end{defn}
\begin{remark}
Note that a priori the target of $F_\mathcal{U}$ should be $\Mod\mathcal{U}$. However, since $\mathcal{C}=\mathcal{U}*\Sigma\mathcal{U}*\dots*\Sigma^{m-1}\mathcal{U}$ by \cite[Theorem 3.1]{IY}, any object $C\in\mathcal{C}$ appears in a triangle of the form
\begin{align*}
    \Sigma^{-1}B\rightarrow A\rightarrow C\rightarrow B,
\end{align*}
where $A\in\mathcal{U}*\Sigma\mathcal{U}$ and $B\in\Sigma^2\mathcal{U}*\dots*\Sigma^{m-1}\mathcal{U}$. Applying $F_\mathcal{U}$ to this triangle, we obtain the exact sequence
\begin{align*}
    0\rightarrow F_{\mathcal{U}}(A)\xrightarrow{\cong} F_\mathcal{U}(C)\rightarrow 0,
\end{align*}
where $F_{\mathcal{U}}(A)\in\mmod\mathcal{U}$ by \cite[Proposition 6.2(3)]{IY} and so $F_{\mathcal{U}}(C)\in \mmod\mathcal{U}$.
\end{remark}

\begin{defn}
If $\mathcal{A}$ and $\mathcal{B}$ are full subcategories of $\mathcal{C}$, then
\begin{align*}
    \mathcal{A}*\mathcal{B}=\{ C\in\mathcal{C}\mid \text{there is a triangle } A\rightarrow C\rightarrow B\rightarrow \Sigma A \text{ with } A\in\mathcal{A},\, B\in\mathcal{B} \}.
\end{align*}
\end{defn}

We will often use towers of triangles, as defined below. These were first introduced in \cite[Notation 3.2]{IY}, see also {\cite[Definition 3.1]{JT}}.
\begin{defn}
A \textit{tower of triangles in $\mathcal{C}$} is a diagram of the form
\begin{align*}
\begin{gathered}
\xymatrix@!0 @C=3em @R=3em{
& C_{l-1}\ar[rd]\ar[rr] && C_{l-2}\ar[rd]\ar[r]&&\cdots&\ar[r]& C_2\ar[rr]\ar[rd]&&C_1\ar[rd]
\\
C_l\ar[ru]&& X_{l-2}\ar@{~>}[ll]\ar[ru]&& X_{l-3}\ar@{~>}[ll]&\cdots& X_2\ar[ru]&& X_1 \ar@{~>}[ll]\ar[ru]&& C_0,\ar@{~>}[ll]
}
\end{gathered}
\end{align*}
where $l\geq 1$ is an integer, a wavy arrow $\xymatrix{X\ar@{~>}[r]&Y}$ signifies a morphism $X\rightarrow \Sigma Y$, each oriented triangle is a triangle in $\mathcal{C}$ and each non-oriented triangle is commutative.
\end{defn}

\begin{defn}[{\cite[Definition 3.3]{JT}}]\label{defn_index}
By \cite[Corollary 3.3]{IY}, for $C\in\mathcal{C}$ there is a tower of triangles in $\mathcal{C}$ of the form
\begin{align*}
\begin{gathered}
\xymatrix@!0 @C=3em @R=3em{
& U_{m-2}\ar[rd]^{\mu_{m-2}}\ar[rr] && U_{m-3}\ar[rd]^{\mu_{m-3}}\ar[r]&&\cdots&\ar[r]& U_1\ar[rr]\ar[rd]^{\mu_{1}}&&U_0 \ar[rd]^{\mu_{0}}
\\
U_{m-1}\ar[ru]&& X_{m-2}\ar@{~>}[ll]\ar[ru]&& X_{m-3}\ar@{~>}[ll]&\cdots& X_2\ar[ru]&& X_1 \ar@{~>}[ll]\ar[ru]&& C,\ar@{~>}[ll]
}
\end{gathered}
\end{align*}
where  $U_i\in\mathcal{U}$ and $\mu_i$ is a $\mathcal{U}$-cover for each $i$. In particular, the $U_i$ are determined up to isomorphism. The \textit{index of $C$ with respect to $\mathcal{U}$} is the following element of the Grothendieck group $K_0^{sp}(\mathcal{U})$:
\begin{align*}
    \text{index}_{\mathcal{U}}(C)=\sum_{i=0}^{m-1}(-1)^i[U_i].
\end{align*}
\end{defn}

\begin{remark}
Note that $\text{index}_{\mathcal{U}}:\text{Obj}(\mathcal{C})\rightarrow K_0^{sp}(\mathcal{U})$ induces a homomorphism $K_0^{sp}(\mathcal{C})\rightarrow K_0^{sp}(\mathcal{U})$ which we also denote by $\text{index}_{\mathcal{U}}$.
\end{remark}
 
\begin{remark}\label{rmk_gen_version}
Note that we work in a more general setup than \cite{JT}, as we dropped the assumption that $\mathcal{U}=\text{add }U$ for some $U\in\mathcal{C}$, or in other words that $\mathcal{U}$ has finitely many indecomposables up to isomorphism. The arguments from \cite{JT} can be easily adjusted in this more general setup and the main results still hold. In particular, we state the corresponding definition of the homomorphism $\theta$ from \cite[Definition 4.1]{JT} and \cite[Theorem 4.5]{JT} in our setup.
\end{remark} 
\begin{defn}\label{defn_theta_U}
There is a homomorphism 
\begin{align*}
    \theta_{\mathcal{U}}: K_0(\mmod \mathcal{U})\rightarrow K_0^{sp}(\mathcal{U}),
\end{align*}
defined by
\begin{align*}
    \theta_{\mathcal{U}}([F_{\mathcal{U}}N])=\text{index}_{\mathcal{U}}(\Sigma^{-1}N)+\text{index}_{\mathcal{U}}(N)
\end{align*}
for $N\in\mathcal{U}*\Sigma \mathcal{U}$.
\end{defn}
\begin{remark}
Note that the fact that $\theta_{\mathcal{U}}$ from Definition \ref{defn_theta_U} is well-defined can be proved using the equivalence of categories $(\mathcal{U}*\Sigma \mathcal{U})/[\Sigma\mathcal{U}]\cong \mmod\mathcal{U}$ from \cite[proposition 6.2(3)]{IY} and the general versions of \cite[lemmas 4.3, 4.4]{JT}, see \cite[Remark 4.2]{JT}.
\end{remark}
As mentioned in Remark \ref{rmk_gen_version}, the argument proving \cite[Theorem 4.5]{JT} can be adjusted to prove the same result in our setup. We state it here for the convenience of the reader.

\begin{proposition}[{\cite[Theorem 4.5]{JT}}]\label{thm_JT_45}
If $A\xrightarrow{\alpha} B\xrightarrow{\beta} C\xrightarrow{\gamma}\Sigma A$ is a triangle in $\mathcal{C}$, then
\begin{align*}
    \textup{index}_\mathcal{U}(B)= \textup{index}_\mathcal{U}(A)+\textup{index}_\mathcal{U}(C)-\theta_\mathcal{U}(\im F_\mathcal{U}(\gamma)).
\end{align*}
\end{proposition}

\section{Morphisms between Grothendieck groups}\label{section_Gr_maps}

\begin{defn}
There are surjective homomorphisms given by the quotient maps
\begin{align*}
    \pi_{\mathcal{C}}: K_0^{sp}(\mathcal{C})\longrightarrow K_0(\mathcal{C}),\,\,
    \pi_{\mathcal{U}}: K_0^{sp}(\mathcal{U})\longrightarrow K_0^{sp}(\mathcal{U})/\im\theta_{\mathcal{U}},
\end{align*}
and injective homomorphisms given by the inclusions
\begin{align*}
    \iota_{\mathcal{C}}:\Ker\pi_\mathcal{C}\rightarrow K_0^{sp}(\mathcal{C}),\,\, \iota_{\mathcal{U}}:\Ker\pi_{\mathcal{U}}\rightarrow K_0^{sp}(\mathcal{U}) \text{ and } j_\mathcal{U}:K_0^{sp}(\mathcal{U})\rightarrow K_0^{sp}(\mathcal{C}).
\end{align*}
\end{defn}

\begin{remark}
Consider the diagram:
\begin{align}\label{diagram_morphisms_f_U}
\begin{gathered}
    \xymatrix@!0@R=6em@C=18em{
    \Ker \pi_{\mathcal{C}}\ar@{^{(}->}[d]_{\iota_{\mathcal{C}}} &\Ker \pi_{\mathcal{U}}\ar@{^{(}->}[d]^{\iota_{\mathcal{U}}}\\
    K_0^{sp}(\mathcal{C})\ar@/^1pc/[r]^-{\text{index}_{\mathcal{U}}}\ar@{->>}[d]_{\pi_{\mathcal{C}}}&K_0^{sp}(\mathcal{U})\ar@/^1pc/[l]^{j_{\mathcal{U}}}\ar@{->>}[d]^{\mathcal{\pi_{\mathcal{U}}}}\\
    K_0(\mathcal{C})\ar@/^1pc/[r]^-{f_{\mathcal{U}}} & K_0^{sp}(\mathcal{U})/\im\theta_{\mathcal{U}}.\ar@{-->}@/^1pc/[l]^-{\exists\, g_{\mathcal{U}} \,?}
    }
\end{gathered}    
\end{align}
It is clear that $\text{index}_{\mathcal{U}}\circ j_{\mathcal{U}}=1_{K_0^{sp}(\mathcal{U})}$ and so $\pi_{\mathcal{U}}\circ \text{index}_{\mathcal{U}}\circ j_{\mathcal{U}}=\pi_{\mathcal{U}}$.
We will show that $\pi_{\mathcal{C}}\circ j_{\mathcal{U}}\circ \text{index}_{\mathcal{U}}=\pi_{\mathcal{C}}$ and that there exists a morphism $f_{\mathcal{U}}: K_0(\mathcal{C})\rightarrow K_0^{sp}(\mathcal{U})/\im\theta_{\mathcal{U}}$ such that
\begin{align*}
    f_{\mathcal{U}}\circ\pi_{\mathcal{C}}=\pi_{\mathcal{U}}\circ\text{index}_{\mathcal{U}}.
\end{align*}
Moreover, adding some assumptions on $\mathcal{C}$ and/or $\mathcal{U}$, we will prove that there exists a morphism
\begin{align*}
    g_{\mathcal{U}}: K_0^{sp}(\mathcal{U})/\im\theta_{\mathcal{U}}\rightarrow K_0(\mathcal{C})
\end{align*}
such that $g_{\mathcal{U}}\circ \pi_{\mathcal{U}}=\pi_{\mathcal{C}}\circ j_{\mathcal{U}}$. In this case, $f_{\mathcal{U}}$ and $g_{\mathcal{U}}$ become inverse isomorphisms.
In the next sections, we consider different sets of extra assumptions under which such a $g_{\mathcal{U}}$ exists.
\end{remark}

\begin{lemma}\label{lemma_pi_ind_j}
We have that
$\pi_{\mathcal{C}}\circ j_{\mathcal{U}}\circ \text{index}_{\mathcal{U}}=\pi_{\mathcal{C}}$.
\end{lemma}

\begin{proof}
Given any object $C\in\mathcal{C}$, consider the tower of triangles from Definition \ref{defn_index}. We have that
\begin{align*}
    \text{index}_{\mathcal{U}}(C)=\sum_{i=0}^{m-1}(-1)^i[U_i].
\end{align*}
Using the relations in $K_0(\mathcal{C})$ corresponding to the triangles in the tower, we have that
\begin{align*}
    \pi_{\mathcal{C}}\circ j_{\mathcal{U}}\circ \text{index}_{\mathcal{U}} ([C])=\pi_{\mathcal{C}}\Bigg(\sum_{i=0}^{m-1}(-1)^i[U_i]\Bigg)=[C]=\pi_{\mathcal{C}}([C]).
\end{align*}
Since this is true for arbitrary $C\in\mathcal{C}$, we conclude that $\pi_{\mathcal{C}}\circ j_{\mathcal{U}}\circ \text{index}_{\mathcal{U}}=\pi_{\mathcal{C}}$.
\end{proof}

\begin{lemma}\label{lemma_f_U}
There is a homomorphism $f_{\mathcal{U}}: K_0(\mathcal{C})\rightarrow K_0^{sp}(\mathcal{U})/\im\theta_{\mathcal{U}}$ such that
    \begin{align*}
        f_{\mathcal{U}}\circ\pi_{\mathcal{C}}=\pi_{\mathcal{U}}\circ\text{index}_{\mathcal{U}}.
    \end{align*}
\end{lemma}

\begin{proof}
There exists a homomorphism $f_{\mathcal{U}}$ with the desired property if and only if $\pi_{\mathcal{U}}\circ\text{index}_{\mathcal{U}}\circ\iota_{\mathcal{C}}=0$.
Note that
\begin{align*}
    \Ker\pi_{\mathcal{C}}=\langle  [A]-[B]+[C] \mid A\rightarrow B\rightarrow C\xrightarrow{\gamma}\Sigma A \text{ is a triangle in }\mathcal{C} \rangle.
\end{align*}
For any generator $[A]-[B]+[C]$ of $\Ker\pi_{\mathcal{C}}$ corresponding to a triangle $A\rightarrow B\rightarrow C\xrightarrow{\gamma}\Sigma A$ in $\mathcal{C}$, we have that
\begin{align*}
    \pi_{\mathcal{U}}\circ\text{index}_{\mathcal{U}}\circ\iota_{\mathcal{C}} ([A]-[B]+[C])&=
    \pi_{\mathcal{U}}(\text{index}_{\mathcal{U}}(A)-\text{index}_{\mathcal{U}}(B)+\text{index}_{\mathcal{U}}(C))\\
    &= \pi_{\mathcal{U}}(\theta_{\mathcal{U}}([\im F_{\mathcal{U}}(\gamma)]))=0,
\end{align*}
where the second equality is obtained by Proposition \ref{thm_JT_45}. Hence $\pi_{\mathcal{U}}\circ\text{index}_{\mathcal{U}}\circ\iota_{\mathcal{C}}=0$ as desired.
\end{proof}

\begin{proposition}\label{prop_if_g_exists}
Suppose there exists a homomorphism $g_{\mathcal{U}}:K_0^{sp}(\mathcal{U})/\im\theta_{\mathcal{U}}\rightarrow K_0(\mathcal{C})$ such that $g_{\mathcal{U}}\circ\pi_{\mathcal{U}}=\pi_{\mathcal{C}}\circ j_{\mathcal{U}}$. Then $f_{\mathcal{U}}$ and $g_{\mathcal{U}}$ are mutually inverse and 
\begin{align*}
    K_0^{sp}(\mathcal{U})/\im\theta_{\mathcal{U}}\cong K_0(\mathcal{C}).
\end{align*}
\end{proposition}

\begin{proof}
Using Lemmas \ref{lemma_pi_ind_j} and \ref{lemma_f_U} and $g_\mathcal{U}$ with the stated property, we have
\begin{align*}
    f_{\mathcal{U}}\circ g_{\mathcal{U}}\circ\pi_{\mathcal{U}}&=f_{\mathcal{U}}\circ \pi_{\mathcal{C}}\circ j_{\mathcal{U}} =\pi_{\mathcal{U}}\circ\text{index}_{\mathcal{U}}\circ j_{\mathcal{U}}=\pi_{\mathcal{U}}= 1_{K_0^{sp}(\mathcal{U})/\im\theta_{\mathcal{U}}}\circ\pi_{\mathcal{U}},\\
    g_{\mathcal{U}}\circ f_{\mathcal{U}}\circ \pi_{\mathcal{C}}&= g_{\mathcal{U}}\circ \pi_{\mathcal{U}}\circ \text{index}_{\mathcal{U}}= \pi_{\mathcal{C}}\circ j_{\mathcal{U}}\circ \text{index}_{\mathcal{U}}
   =\pi_{\mathcal{C}}=1_{K_0(\mathcal{C})}\circ\pi_{\mathcal{C}}.
    \end{align*}
Since $\pi_{\mathcal{U}}$ and $\pi_{\mathcal{C}}$ are surjective, and hence right cancellative, we have
\begin{align*}
    f_{\mathcal{U}}\circ g_{\mathcal{U}}=1_{K_0^{sp}(\mathcal{U})/\im\theta_{\mathcal{U}}} \,\,\, \text{ and }\,\,\, g_{\mathcal{U}}\circ f_{\mathcal{U}}= 1_{K_0(\mathcal{C})}.
\end{align*}
\end{proof}

\section{$\mathcal{C}$ with Serre functor and $n$-cluster tilting subcategory $\mathcal{T}$}\label{section_thm_T}
\begin{defn}[{\cite[Section 1.1]{JKu}}]\label{defn_lb}
Let $\mathcal{T}\subseteq\mathcal{C}$ be a full subcategory and Ind$\,\mathcal{T}$ be a full subcategory of $\mathcal{T}$ containing precisely one object from each isomorphism class of indecomposable objects in $\mathcal{T}$. We say that Ind$\,\mathcal{T}$ is \textit{locally bounded} if for every  object $T\in\text{Ind}\,\mathcal{T}$, there are only finitely many objects $V\in\text{Ind}\,\mathcal{T}$ such that $\mathcal{C}(U,V)\neq 0$ and only finitely many objects $W\in\text{Ind}\,\mathcal{T}$ such that $\mathcal{C}(W,U)\neq 0$.
\end{defn}
\begin{setup}\label{setup_4}
Assume that $\mathcal{C}$ has a Serre functor $\mathbb{S}$.
Let $n\geq 2$ be an integer and let $\mathcal{T}\subseteq\mathcal{C}$ be an $n$-cluster tilting subcategory such that $\text{Ind } \mathcal{T}$ is locally bounded.
\end{setup}

\begin{remark}\label{rmk_Sn}
Using the same notation as in \cite[Section 3]{IY}, we define the functor $\mathbb{S}_n:=\mathbb{S}\circ\Sigma^{-n}$ on $\mathcal{C}$.
Let $M$ be an indecomposable in $\mathcal{T}$. By \cite[Theorem 3.10]{IY}, there is an Auslander--Reiten $(n+2)$-angle in $\mathcal{T}$, as defined in \cite[Definition 3.8]{IY}, given by a tower of triangles in $\mathcal{C}$ of the form:
\begin{align}\label{diagram_AR}
\begin{gathered}
\xymatrix@!0 @C=3em @R=3em{
& T_{n-1}\ar[rd]^{\tau_{n-1}}\ar[rr] && T_{n-2}\ar[rd]^{\tau_{n-2}}\ar[r]&&\cdots&\ar[r]& T_1\ar[rr]\ar[rd]^{\tau_1} &&T_0\ar[rd]^{\tau_0}
\\
\mathbb{S}_n(M)\ar[ru]&& X_{n-1}\ar@{~>}[ll]^{\xi_{n-1}}\ar[ru]&& X_{n-2}\ar@{~>}[ll]^{\xi_{n-2}}&\cdots& X_2\ar[ru]&& X_1 \ar@{~>}[ll]^{\xi_1}\ar[ru]&& M.\ar@{~>}[ll]^{\xi_0}
}
\end{gathered}
\end{align}
Note that $M,\, \mathbb{S}_n(M),\,T_0,\,\dots,\, T_{n-1}\in\mathcal{T}$.
\end{remark}

\begin{lemma}\label{lemma_F_xi_zero}
Let $M\in\mathcal{T}$ be an indecomposable with Auslander--Reiten $(n+2)$-angle as in (\ref{diagram_AR}). Then $F_{\mathcal{T}}(\xi_i)=0$ for any $i=1,\dots,n-1$.
\end{lemma}
\begin{proof}
By \cite[Definition 3.8]{IY}, $\tau_i: T_{i}\rightarrow X_i$ is a $\mathcal{T}$-cover of $X_i$. Hence, for every every object $\overline{T}\in\mathcal{T}$ and every morphism $\tau\in \mathcal{C}( \overline{T},X_i)$, there is a morphism $\tau':\overline{T}\rightarrow T_i$ such that $\tau=\tau_i\circ\tau'$. Then,
\begin{align*}
((F_{\mathcal{T}}(\xi_i))(\overline{T}))(\tau)=\xi_i\circ\tau=\xi_i\circ\tau_i\circ\tau'=0,
\end{align*}
where $\xi_i\circ\tau_i=0$ because two consecutive morphisms in a triangle compose to zero. Since this is true for arbitrary $\overline{T}\in\mathcal{T}$ and $\tau\in \mathcal{C}(\overline{T},X_i)$, we conclude that $F_{\mathcal{T}}(\xi_i)=0$ for any $i=1,\dots,n-1$.
\end{proof}

\begin{lemma}\label{lemma_theta_S_M}
Let $M\in\mathcal{T}$ be an indecomposable and consider diagram (\ref{diagram_AR}). Then, as an element in $K_0^{sp}(\mathcal{T})$, we have
\begin{align*}
    -[M]+(-1)^n[\mathbb{S}_n(M)]+[T_0]-[T_1]+\dots +(-1)^{n-1}[T_{n-1}]=-\theta_{\mathcal{T}} ([S_M]),
\end{align*}
where $S_M$ is the simple in $\mmod\mathcal{T}$ that is the top of $\mathcal{C}(-,M)|_{\mathcal{T}}$, the projective in $\mmod\mathcal{T}$ corresponding to $M$.
\end{lemma}

\begin{proof}
Consider the exact sequence induced by the rightmost triangle in (\ref{diagram_AR}):
\begin{align*}
    \mathcal{C}(-,T_0)|_{\mathcal{T}}\xrightarrow{F_{\mathcal{T}}(\tau_0)}\mathcal{C}(-,M)|_{\mathcal{T}}\xrightarrow{F_{\mathcal{T}}(\xi_0)} \mathcal{C}(-,\Sigma X_1)|_{\mathcal{T}}\rightarrow \mathcal{C}(-,\Sigma T_0)|_{\mathcal{T}}.
\end{align*}
Note that $\mathcal{C}(-,\Sigma T_0)|_{\mathcal{T}}=0$ and so $\im F_{\mathcal{T}}(\xi_0)=\Coker F_{\mathcal{T}}(\tau_0) =\mathcal{C}(-,\Sigma X_1)|_{\mathcal{T}}.$
By \cite[Definition 3.8]{IY}, we have that $\tau_0: T_0\rightarrow M$ is minimal right almost split in $\mathcal{T}$ and so using \cite[Corollary 2.5]{AMr}, we have that
\begin{align*}
    S_M = \mathcal{C}(-,\Sigma X_1)|_\mathcal{T}
\end{align*}
is the simple in $\mmod\mathcal{T}$ that is the top of  $\mathcal{C}(-,M)|_\mathcal{T}$. Then, by Proposition \ref{thm_JT_45}, we have
\begin{align}\label{eqn_T_0}
    [T_0]=\text{index}_{\mathcal{T}}(T_0)=[M]+\text{index}_{\mathcal{T}}(X_1)-\theta_{\mathcal{T}}([S_M]).
\end{align}
Moreover, since $\tau_1,\dots,\, \tau_{n-1}$ are $\mathcal{T}$-covers, by Definition \ref{defn_index} we have
\begin{align*}
    \text{index}_{\mathcal{T}}(X_1)=\sum_{i=1}^{n-1} (-1)^{i-1}[T_i]+(-1)^{n-1}[\mathbb{S}_n(M)].
\end{align*}
Substituting this into (\ref{eqn_T_0}), we conclude that
\begin{align*}
   -[M]+(-1)^n[\mathbb{S}_n(M)]+[T_0]-[T_1]+\dots +(-1)^{n-1}[T_{n-1}]=-\theta_{\mathcal{T}} ([S_M]).
\end{align*}
\end{proof}

\begin{remark}\label{remark_im_theta_gen}
Note that Lemma \ref{lemma_theta_S_M} can be applied to any indecomposable in $\mathcal{T}$. Moreover, since $\text{Ind }\mathcal{T}$ is locally bounded, then each object in $\mmod\mathcal{T}$ has finite length and $K_0(\mmod \mathcal{T})$ is generated by the equivalence classes of the simples in  $\mmod\mathcal{T}$. Since any simple object in $\mmod\mathcal{T}$ has the form $S_M$ for some indecomposable $M\in\mathcal{T}$, see \cite[Sections 3.1 and 3.2]{GR}, we have
\begin{align*}
    \im \theta_{\mathcal{T}}=\Bigg\langle
    \xymatrix{ 
    -[M]+(-1)^n[\mathbb{S}_n(M)]+\displaystyle\sum_{i=0}^{n-1} (-1)^i [T_i]
   \,\,\, \Bigg\vert \,\,\,{\begin{matrix}M\in\rm{Ind }\,\mathcal{T} \text{ with Auslander--Reiten }\\ (n+2)\text{-angle } (\ref{diagram_AR})\end{matrix}}}  \Bigg\rangle.
\end{align*}
\end{remark}

\begin{lemma}\label{lemma_f_g_exist}
There is a morphism $g_{\mathcal{T}}: K_0^{sp}(\mathcal{T})/\im\theta_{\mathcal{T}}\rightarrow K_0(\mathcal{C})$ such that
    \begin{align*}
        g_{\mathcal{T}}\circ\pi_{\mathcal{T}}=\pi_{\mathcal{C}}\circ j_{\mathcal{T}}.
    \end{align*}
\end{lemma}

\begin{proof}
Consider diagram (\ref{diagram_morphisms_f_U}) with $\mathcal{U}=\mathcal{T}$.
A morphism $g_\mathcal{U}$ with the desired property exists if and only if $\pi_{\mathcal{C}}\circ j_{\mathcal{T}}\circ\iota_{\mathcal{T}}=0$. 
By Remark \ref{remark_im_theta_gen}, we have
\begin{align*}
    \Ker\pi_{\mathcal{T}}=\Bigg\langle
    \xymatrix{ 
    -[M]+(-1)^n[\mathbb{S}_n(M)]+\displaystyle\sum_{i=0}^{n-1} (-1)^i [T_i]
   \,\,\, \Bigg\vert \,\,\,{\begin{matrix}M\in\rm{Ind }\,\mathcal{T} \text{ with Auslander--Reiten }\\ (n+2)\text{-angle } (\ref{diagram_AR})\end{matrix}}}  \Bigg\rangle.
\end{align*}
Then, for any generator $ -[M]+(-1)^n[\mathbb{S}_n(M)]+\sum_{i=0}^{n-1} (-1)^i [T_i]$ of $\Ker\pi_{\mathcal{T}}$ corresponding to the Auslander--Reiten $(n+2)$-angle (\ref{diagram_AR}), we have
\begin{align*}
    \pi_{\mathcal{C}}\circ j_{\mathcal{T}}\circ\iota_{\mathcal{T}}\Big( -[M]+(-1)^n[\mathbb{S}_n(M)]+\sum_{i=0}^{n-1} (-1)^i [T_i]\Big)&= -[M]+(-1)^n[\mathbb{S}_n(M)]+\sum_{i=0}^{n-1} (-1)^i [T_i]\\&=0,
\end{align*}
where all the terms cancel because of the relations in $K_{0}(\mathcal{C})$ corresponding to the triangles in the tower (\ref{diagram_AR}). 
Hence $\pi_{\mathcal{C}}\circ j_{\mathcal{T}}\circ\iota_{\mathcal{T}}=0$ as desired.
\end{proof}

We now prove Theorem A of the introduction.
\begin{proof}[Proof of Theorem A]
By Lemma \ref{lemma_f_g_exist}, there exists a homomorphism $g_{\mathcal{T}}: K_0^{sp}(\mathcal{T})/\im\theta_{\mathcal{T}}\rightarrow K_0(\mathcal{C})$ such that $g_{\mathcal{T}}\circ\pi_{\mathcal{T}}=\pi_{\mathcal{C}}\circ j_{\mathcal{T}}$. Then, by Proposition \ref{prop_if_g_exists}, we have that $K_0(\mathcal{C})\cong K_0^{sp}(\mathcal{T})/\im\theta_{\mathcal{T}}$ and, by Remark \ref{remark_im_theta_gen}, this completes the proof.
\end{proof}

\begin{remark}
Note that our argument does not apply to the case when $n=1$ as it uses some results, such as Proposition \ref{thm_JT_45}, that rely on $n\geq 2$.
However, in the case when $k$ is an algebraically closed field and $n=1$, Theorem A is an instance of the triangulated analogue by Xiao and Zhu of results of Auslander, see \cite[Theorems 2.2 and 2.3]{AM}, and Butler, see \cite[Theorem in introduction]{BMCR}, on certain module categories.
In this case, the only choice of $1$-cluster tilting subcategory is $\mathcal{C}=\mathcal{T}$ and the tower of triangles (\ref{diagram_AR}) is an Auslander--Reiten triangle in $\mathcal{C}$ of the form
\begin{align*}
\delta: \mathbb{S}\Sigma^{-1}(M)\rightarrow T_0\rightarrow M\rightarrow \Sigma \mathbb{S}\Sigma^{-1}(M).
\end{align*}
Note that, since Ind$\,\mathcal{T}$ is locally bounded, we have that $\mathcal{C}=\mathcal{T}$ is of finite type. Then, by \cite[Theorem 2.1]{XZ},
 we have that $K_0(\mathcal{C})$ is isomorphic to the quotient of $K_0^{sp}(\mathcal{C})$ by the elements $[\delta]:=-[M]+[T_0]-[\mathbb{S}\Sigma^{-1}(M)]$, where $\delta$ runs through all the Auslander--Reiten triangles in $\mathcal{C}$.
\end{remark}


Adding the assumptions that $n\geq 2$ is an even integer and $\mathcal{C}$ is $n$-Calabi--Yau to Setup \ref{setup_4}, we prove Corollary B of the introduction.

\begin{proof}[Proof of Corollary B]
Since $\mathcal{C}$ is $n$-Calabi--Yau, it has Serre functor $\mathbb{S}=\Sigma^n$ so this is a special case of Setup \ref{setup_4}. Note also that in this case $\mathbb{S}_n=\mathbb{S}\circ \Sigma^{-n}$ is the identity functor on $\mathcal{C}$. Hence, since $n$ is even, we have that $-[M]+(-1)^n[\mathbb{S}_n(M)]=-[M]+[M]=0$ and the result follows from Theorem A.
\end{proof}

\begin{remark}
 Note that when $n=2$ and the Auslander--Reiten quiver of $\mathcal{T}$ has no loops, Corollary B coincides with \cite[Theorem 10]{PY}. In this case,
if $M$ is an indecomposable direct summand of $T$, then its Auslander--Reiten $4$-angle is $M\rightarrow B_{M^*}\rightarrow B_M\rightarrow M\rightarrow \Sigma^2 M$, where $B_M,\,B_{M^*}$ are defined as in \cite[p. 1444]{PY}.
However, if we do not assume that the Auslander--Reiten quiver of $\mathcal{T}$ has no loops, some of the Auslander--Reiten $4$-angles in $\mathcal{T}$ do not come from Palu's exchange triangles and Corollary B and \cite[Theorem 10]{PY} are different.
\end{remark}

\section{A $\Sigma^d$-stable, $d$-cluster tilting subcategory $\mathcal{S}\subseteq \mathcal{C}$}\label{section_thm_S}
\begin{setup}
Let $d\geq 1$ be an integer and $\mathcal{S}\subseteq \mathcal{C}$ be a $d$-cluster tilting subcategory. Assume also that $\Sigma^d \mathcal{S}=\mathcal{S}$. Then, by \cite[Theorem 1]{GKO}, we have that $\mathcal{S}$ is a $(d+2)$-angulated category with $d$-suspension functor $\Sigma^d$.
\end{setup}

\begin{lemma}\label{lemma_F_S_zero}
Consider a $(d+2)$-angle in $\mathcal{S}$ of the form
\begin{align*}
    \xymatrix{
    S_{d+1}\ar[r]& S_{d}\ar[r] &\cdots\ar[r]&S_2\ar[r]& S_1\ar[r]&S_0\ar[r]&\Sigma^d S_{d+1}.
    }
\end{align*}
By \cite[theorem 1]{GKO}, it corresponds to a tower of triangles in $\mathcal{C}$:
\begin{align}\label{diagram_tower_S}
\begin{gathered}
\xymatrix@!0 @C=3em @R=3em{
& S_{d}\ar[rd]^{\sigma_d}\ar[rr] && S_{d-1}\ar[rd]^{\sigma_{d-1}}\ar[r]&&\cdots&\ar[r]& S_3\ar[rr]\ar[rd]^{\sigma_3} &&S_2\ar[rd]^{\sigma_2}\ar[rr]&& S_1\ar[rd]
\\
S_{d+1}\ar[ru]&& Y_{d-1}\ar@{~>}[ll]^{\eta_{d-1}}\ar[ru]&& Y_{d-2}\ar@{~>}[ll]^{\eta_{d-2}}&\cdots& Y_3\ar[ru]&& Y_2 \ar@{~>}[ll]^{\eta_2}\ar[ru]&& Y_1\ar@{~>}[ll]^{\eta_1}\ar[ru]&& S_0.\ar@{~>}[ll]^{\eta_0}
}
\end{gathered}
\end{align}
This tower satisfies $F_\mathcal{S}(\eta_l)=0$ for any integer $1\leq l\leq d-1$.
\end{lemma}

\begin{proof}
Let $Y_d:=S_{d+1}$. In order to prove that $F_{\mathcal{S}}(\eta_l)=0$, we prove that its target $F_{\mathcal{S}}(\Sigma Y_{l+1})$ is zero. More generally, we prove that $F_{\mathcal{S}}(\Sigma^i Y_{d-j})=0$ for any integers $0\leq j\leq d-2$ and $1\leq i\leq d-j-1$.

First note that for $j=0$, we have $F_{\mathcal{S}}(\Sigma^i Y_d)=0$ for any $1\leq i\leq d-1$, since $Y_d=S_{d+1}\in\mathcal{S}$.

Suppose that for some $0\leq j\leq d-2$, we proved that
\begin{align}
    F_{\mathcal{S}}(\Sigma^i Y_{d-k})=0\,\, \text{ for any }\,\, 0\leq k\leq j\,\, \text{ and }\,\, 1\leq i\leq d-k-1. \tag{\dag}
\end{align}
If $j=d-2$, then we are done. So assume $j\leq d-3$. We need to prove that
\begin{align*}
    F_{\mathcal{S}}(\Sigma^{i'} Y_{d-(j+1)})=0 \,\, \text{ for any }\,\, 1\leq i'\leq d-j-2.
\end{align*}
Given any $1\leq i'\leq d-j-2$, the triangle $S_{d-j}\rightarrow Y_{d-(j+1)}\rightarrow \Sigma Y_{d-j}\rightarrow \Sigma S_{d-j}$ induces the exact sequence:
\begin{align*}
    F_{\mathcal{S}}(\Sigma^{i'} S_{d-j})\rightarrow F_{\mathcal{S}}(\Sigma^{i'} Y_{d-(j+1)})\rightarrow F_{\mathcal{S}}(\Sigma^{i'+1} Y_{d-j}).
\end{align*}
Note that $F_{\mathcal{S}}(\Sigma^{i'} S_{d-j})=0$ since $S_{d-j}\in\mathcal{S}$ and $1\leq i'<d-1$. Moreover, $F_{\mathcal{S}}(\Sigma^{i'+1} Y_{d-j})=0$ by ($\dag$) with $k=j$ and as $1<i'+1\leq d-j-1$. Hence $F_{\mathcal{S}}(\Sigma^{i'} Y_{d-(j+1)})=0$ for any $1\leq i'\leq d-j-2$ as we wished to show.
\end{proof}

\begin{remark}\label{remark_no_covers}
Consider a $(d+2)$-angle in $\mathcal{S}$ of the form
\begin{align*}
    \xymatrix{
    S_{d+1}\ar[r]& S_{d}\ar[r] &\cdots\ar[r]&S_2\ar[r]& S_1\ar[r]&S_0\ar[r]&\Sigma^d S_{d+1}.
    }
\end{align*}
By \cite[Theorem 1]{GKO}, it corresponds to a tower of triangles in $\mathcal{C}$ of the form (\ref{diagram_tower_S}). By Lemma \ref{lemma_F_S_zero}, we have that $F_\mathcal{S}(\eta_l)=0$ for any integer $1\leq l\leq d-1$. Hence $\sigma_{l+1}:S_{l+1}\rightarrow Y_{l}$ is an $\mathcal{S}$-precover. If it is not right minimal, then $S_{l+1}\cong \overline{S}_{l+1}\oplus S_{l+1}'$ and $\sigma_{l+1}$ is isomorphic to a morphism of the form $(0,\sigma_{l+1}'):\overline{S}_{l+1}\oplus S_{l+1}'\rightarrow Y_l$, where $\sigma_{l+1}'$ is an $\mathcal{S}$-cover of $Y_l$. It is then easy to check that $\overline{S}_{l+1}$ appears as a summand of $S_{l+2}$ and so it gets cancelled when computing $[S_{l+1}]-[S_{l+2}]$.
Hence, the morphisms $\sigma_l$ do not need to be $\mathcal{S}$-covers for using the tower (\ref{diagram_tower_S}) to compute $\text{index}_{\mathcal{S}}(Y_1)$. In other words,
\begin{align*}
    \text{index}_{\mathcal{S}}(Y_1)=\sum_{i=2}^{d+1}(-1)^{i}[S_i].
\end{align*}
\end{remark}

\begin{proposition}\label{proposition_im_theta_S}
When $d\geq2$, we have that
\begin{align*}
    \im\theta_{\mathcal{S}}=\Bigg\langle \sum_{i=0}^{d+1} (-1)^i [S_i]\mid S_{d+1}\rightarrow\dots\rightarrow S_0\rightarrow\Sigma^d S_{d+1} \text{ is a $(d+2)$-angle in } \mathcal{S} \Bigg\rangle.
\end{align*}
\end{proposition}

\begin{proof}
We prove this by proving that the two inclusions hold.

($\subseteq$). Given any $Y\in\mathcal{S}*\Sigma\mathcal{S}$, there is a triangle in $\mathcal{C}$ of the form
\begin{align*}
    \xymatrix{
    S_0\ar[r]^-{\eta_0}& Y\ar[r]&\Sigma S_1\ar[r]&\Sigma S_0,
    }
\end{align*}
where $S_0,\,S_1\in\mathcal{S}$. Letting $Y_1:=\Sigma^{-1}Y\in\mathcal{C}$, we obtain a triangle in $\mathcal{C}$ of the form
\begin{align*}
    \xymatrix{
   \Delta:& Y_1\ar[r]& S_1\ar[r]&S_0\ar[r]^-{\eta_0}& \Sigma Y_1.
    }
\end{align*}
Since $\mathcal{S}$ is $d$-cluster tilting in $\mathcal{C}$, by \cite[corollary 3.3]{IY}, we can construct a tower of triangles in $\mathcal{C}$ of the form
\begin{align*}
\begin{gathered}
\xymatrix@!0 @C=3em @R=3em{
& S_{d}\ar[rd]\ar[rr] && S_{d-1}\ar[rd]\ar[r]&&\cdots&\ar[r]& S_3\ar[rr]\ar[rd] &&S_2\ar[rd]
\\
S_{d+1}\ar[ru]&& Y_{d-1}\ar@{~>}[ll]^{\eta_{d-1}}\ar[ru]&& Y_{d-2}\ar@{~>}[ll]^{\eta_{d-2}}&\cdots& Y_3\ar[ru]&& Y_2 \ar@{~>}[ll]^{\eta_2}\ar[ru]&& Y_1,\ar@{~>}[ll]^{\eta_1}
}
\end{gathered}
\end{align*}
where $S_2,\dots,\,S_{d+1}$ are in $\mathcal{S}$. Putting this together with triangle $\Delta$, we obtain the tower of triangles (\ref{diagram_tower_S}) in $\mathcal{C}$,
which corresponds to the $(d+2)$-angle in $\mathcal{S}$:
\begin{align*}
    \xymatrix{
    S_{d+1}\ar[r]& S_{d}\ar[r] &\cdots\ar[r]&S_2\ar[r]& S_1\ar[r]&S_0\ar[r]&\Sigma^d S_{d+1}.
    }
\end{align*}
By Proposition \ref{thm_JT_45}, we have that in $K_0^{sp}(\mathcal{S})$:
\begin{align*}
    [S_1]=\text{index}_{\mathcal{S}}(Y_1)+[S_0]-\theta_{\mathcal{S}}([\im F_{\mathcal{S}}(\eta_0)]).
\end{align*}
Moreover, since $F_\mathcal{S}(\Sigma S_1)=0$, we have that $F_{\mathcal{S}}(\eta_0)$ is surjective and so
\begin{align*}
    [S_1]=\text{index}_{\mathcal{S}}(Y_1)+[S_0]-\theta_{\mathcal{S}}([ F_{\mathcal{S}}(\Sigma Y_1)]).
\end{align*}
We have that $\text{index}_{\mathcal{S}}(Y_1)=\sum_{i=2}^{d+1}(-1)^{i}[S_i]$ and so
\begin{align*}
    \sum_{i=0}^{d+1}(-1)^i[S_i]=\theta_{\mathcal{S}}([ F_{\mathcal{S}}(\Sigma Y_1)])=\theta_{\mathcal{S}}([ F_{\mathcal{S}}(Y)]).
\end{align*}

($\supseteq$). Given a $(d+2)$-angle in $\mathcal{S}$ of the form
\begin{align*}
    \xymatrix{
    S_{d+1}\ar[r]& S_{d}\ar[r] &\cdots\ar[r]&S_2\ar[r]& S_1\ar[r]&S_0\ar[r]&\Sigma^d S_{d+1},
    }
\end{align*}
consider the corresponding tower (\ref{diagram_tower_S}) of triangles in $\mathcal{C}$.
By Remark \ref{remark_no_covers}, we have that $\text{index}_{\mathcal{S}}(Y_1)=\sum_{i=2}^{d+1}(-1)^{i}[S_i]$. Using Proposition \ref{thm_JT_45} on the rightmost triangle in tower (\ref{diagram_tower_S}), namely $Y_1\rightarrow S_1\rightarrow S_0\xrightarrow{\eta_0} \Sigma Y_1$, we conclude that
\begin{align*}
    \sum_{i=0}^{d+1}(-1)^i[S_i]=\theta_{\mathcal{S}}([ \im F_{\mathcal{S}}(\eta_0)])\in\im\theta_{\mathcal{S}}.
\end{align*}
\end{proof}

\begin{remark}\label{remark_imS_K0}
By Proposition \ref{proposition_im_theta_S} and Definition \ref{defn_GG_S}, for $d\geq 2$, we have that
\begin{align*}
     K_0(\mathcal{S})= K_0^{sp}(\mathcal{S})/\im\theta_{\mathcal{S}}.
\end{align*}
Suppose now that Ind$\,\mathcal{S}$ is locally bounded and $\mathcal{C}$ has a Serre functor $\mathbb{S}$.
By \cite[Theorem 3.10]{IY}, for every indecomposable $M$ in $\mathcal{S}$, there is an Auslander--Reiten $(d+2)$-angle in $\mathcal{S}$ of the form
\begin{align}\label{diagram_AR_remark}
    \xymatrix{
    \mathbb{S}_d(M)\ar[r]& S_{d-1}\ar[r] &\cdots\ar[r]&S_1\ar[r]& S_0\ar[r]&M\ar[r]&\Sigma^d \mathbb{S}_d(M),
    }
\end{align}
where $\mathbb{S}_d=\mathbb{S}\circ\Sigma^{-d}$.
Then, by Remark \ref{remark_im_theta_gen}, we have that
\begin{align*}
    K_0(\mathcal{S})= K_0^{sp}(\mathcal{S})\Bigg/\Bigg\langle
    \xymatrix{ 
    -[M]+(-1)^d[\mathbb{S}_d(M)]+\displaystyle\sum_{i=0}^{d-1} (-1)^i [S_i]
   \,\,\, \Bigg\vert \,\,\,{\begin{matrix}M\in\rm{Ind }\,\mathcal{S} \text{ with Auslander--Reiten }\\ (d+2)\text{-angle } (\ref{diagram_AR_remark})\end{matrix}}}  \Bigg\rangle.
\end{align*}
This result agrees with \cite[Theorem 3.7]{ZP}. Note that there are two differences between ours and Zhou's result. The first one is that we do not assume that $d$ is odd, and the second one is that Zhou's  $(d+2)$-angulated category is not assumed to raise as a $d$-cluster tilting subcategory of a triangulated category.
\end{remark}

\begin{lemma}\label{lemma_a_b_exist}
When $d\geq 2$, there is a morphism $g_{\mathcal{S}}: K_0(\mathcal{S})=K_0^{sp}(\mathcal{S})/\im\theta_{\mathcal{S}}\rightarrow K_0(\mathcal{C})$ such that
    \begin{align*}
        g_{\mathcal{S}}\circ\pi_{\mathcal{S}}=\pi_{\mathcal{C}}\circ j_{\mathcal{S}}.
    \end{align*}
\end{lemma}

\begin{proof}
Consider diagram (\ref{diagram_morphisms_f_U}) with $\mathcal{U}=\mathcal{S}$. Note that a morphism $g_{\mathcal{S}}$ with the desired property exists if and only if $\pi_{\mathcal{C}}\circ j_{\mathcal{S}}\circ\iota_{\mathcal{S}}=0$. Note that, by Proposition \ref{proposition_im_theta_S}, the group $\ker\pi_\mathcal{S}$ is generated by elements of the form
\begin{align*}
    \sum_{i=0}^{d+1} (-1)^i[S_i],
\end{align*}
for some $(d+2)$-angle in $\mathcal{S}$ of the form $S_{d+1}\rightarrow\dots\rightarrow S_0\rightarrow\Sigma^d S_{d+1}$. Such a $(d+2)$-angle corresponds to a tower of triangles in $\mathcal{C}$ of the form (\ref{diagram_tower_S}). Then, we have
\begin{align*}
    \pi_{\mathcal{C}}\circ j_{\mathcal{S}}\circ \iota_{\mathcal{S}} \Big(\sum_{i=0}^{d+1} (-1)^i[S_i]\Big)&= [S_0]-([S_0]+[Y_1])+\cdots+(-1)^{d}([Y_{d-1}]+[S_{d+1}])+(-1)^{d+1}[S_{d+1}]\\&=0,
\end{align*}
where we have used the relations in $K_0(\mathcal{C})$ corresponding to the triangles in the tower (\ref{diagram_tower_S}), for instance $[S_1]=[S_0]+[Y_1]$. Hence $\pi_{\mathcal{C}}\circ j_{\mathcal{S}}\circ \iota_{\mathcal{C}} =0$ as desired.
\end{proof}

We now prove Theorem C of the introduction.
\begin{proof}[Proof of Theorem C]
If $d=1$, then $\mathcal{S}=\mathcal{C}$ and the result is clear. So assume $d\geq 2$.
By Lemma \ref{lemma_a_b_exist}, there exists a homomorphism $g_{\mathcal{S}}: K_0^{sp}(\mathcal{S})/\im\theta_{\mathcal{S}}\rightarrow K_0(\mathcal{C})$ such that $g_{\mathcal{S}}\circ\pi_{\mathcal{S}}=\pi_{\mathcal{C}}\circ j_{\mathcal{S}}$. Then, by Proposition \ref{prop_if_g_exists}, we have that $K_0(\mathcal{C})\cong K_0^{sp}(\mathcal{S})/\im\theta_{\mathcal{S}}$ and, by Remark \ref{remark_imS_K0}, this completes the proof.
\end{proof}
\section{The case when $n=2d$ and $\mathcal{T}\subseteq\mathcal{S}\subseteq \mathcal{C}$}\label{section_coro}
\begin{setup}\label{setup_T_S}
Let $d\geq 1$ be an integer and $n=2d$. Assume that $\mathcal{C}$ has a Serre functor $\mathbb{S}$ and $\mathcal{T}\subseteq\mathcal{S}\subseteq \mathcal{C}$ are such that $\mathcal{T}$ is an $n$-cluster tilting subcategory in $\mathcal{C}$ such that Ind$\,\mathcal{T}$ is locally bounded
and $\mathcal{S}$ is a $d$-cluster tilting subcategory in $\mathcal{C}$ such that $\Sigma^d\mathcal{S}=\mathcal{S}$. Then $\mathcal{S}$ is a $(d+2)$-angulated category with $d$-suspension $\Sigma^d$.
\end{setup}
 
\begin{defn}[{\cite[Definition 5.3]{OT}}]
A functorially finite, full subcategory $\mathcal{T}\subseteq\mathcal{S}$ is an \textit{Oppermann--Thomas cluster tilting subcategory} if:
\begin{enumerate}[label=(\alph*)]
    \item $\mathcal{S}(\mathcal{T}, \Sigma^d \mathcal{T})=0$,
    \item for each $S'\in\mathcal{S}$, there is a $(d+2)$-angle $T_d\rightarrow \dots\rightarrow T_0\rightarrow S'\rightarrow \Sigma^d T_d$ in $\mathcal{S}$ with $T_i\in\mathcal{T}$.
\end{enumerate}
\end{defn}

\begin{remark}
The subcategory $\mathcal{T}\subseteq\mathcal{S}$ from Setup \ref{setup_T_S} is an Oppermann--Thomas cluster tilting subcategory by \cite[Theorem 5.25]{OT}. Note that in \cite[Definition 5.3 and Theorem 5.25]{OT} $\mathcal{T}$ is assumed to have finitely many indecomposables up to isomorphism. We are not restricting to this case and it can be easily checked that the proof of \cite[Theorem 5.25]{OT} still goes through without this restriction.
\end{remark}


\begin{proof}[Proof of Theorem D]
Theorem D of the introduction then follows by combining Theorem A and Theorem C, and noting that $[\mathbb{S}_n(M)]=(-1)^n[\mathbb{S}_n(M)]$ since $n=2d$ is even.
\end{proof}


Adding the assumption that $\mathcal{C}$ is $n$-Calabi--Yau, we prove Corollary E of the introduction.
\begin{proof}[Proof of Corollary E]
Since $\mathcal{C}$ is $n$-Calabi--Yau, it has Serre functor $\mathbb{S}=\Sigma^n$. Then $\mathbb{S}_n$ is the identity functor on $\mathcal{C}$ and the result follows from Theorem D.
\end{proof}

\begin{remark}
When $d=1$, we have that $\mathcal{S}=\mathcal{C}$ is a triangulated category with cluster tilting subcategory $\mathcal{T}$ and, adding the extra assumption that the Auslander--Reiten quiver of $\mathcal{T}$ has no loops, Corollary E becomes \cite[Theorem 10]{PY} by Palu. For higher values of $d$, Corollary E is a higher angulated version of Palu's theorem.
\end{remark}

\section{The Grothendieck group of $\mathcal{C}_q(A_p)$ for $q$ odd}\label{section_eg1}
In this section, we compute the Grothendieck group of the triangulated $q$-cluster category of Dynkin type $A_p$ for $q$ odd. We start by introducing this category, first defined in \cite{TH}, and its geometric model, see \cite{MJ} and \cite{BM} for more details.

Let $q$ and $p$ be positive integers and consider the coordinate system on the translation quiver $\mathbb{Z} A_p$ illustrated in Figure \ref{fig:coordinate_ZAn}.

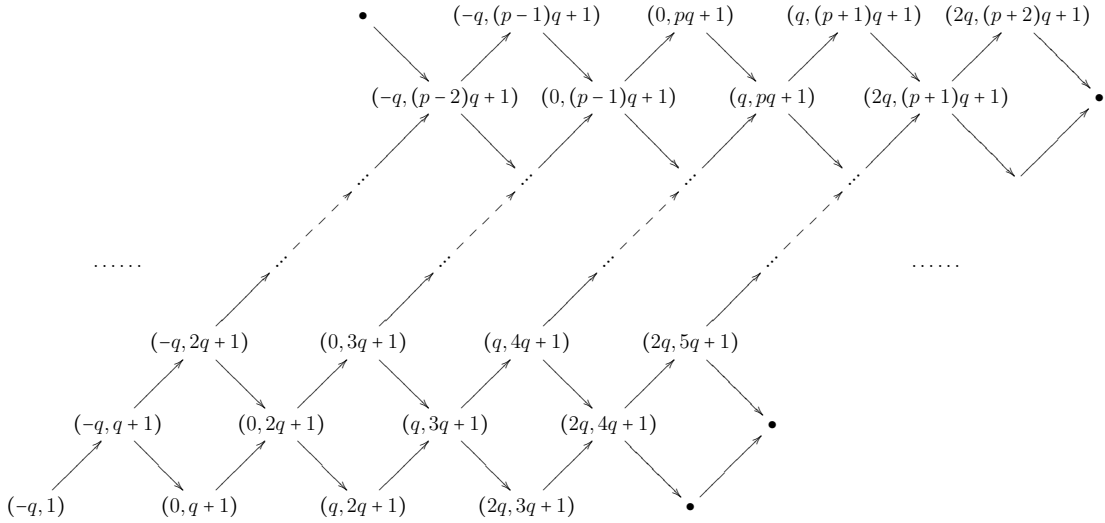
\begin{figure}
\centerline{
\resizebox{0.9\textwidth}{!}{
    \xymatrix @!0@C=4em@R=4em{
     &&&&\bullet\ar[rd]&& (-q,(p-1)q+1)\ar[rd] && (0,pq+1)\ar[rd]&& (q,(p+1)q+1)\ar[rd]&& (2q, (p+2)q+1)\ar[rd]\\ 
    &&&&& (-q,(p-2)q+1)\ar[rd]\ar[ru] && (0,(p-1)q+1)\ar[rd]\ar[ru]&& (q,pq+1)\ar[rd]\ar[ru]&& (2q, (p+1)q+1)\ar[ru]\ar[rd]&&\bullet  \\
    &&&&\udots\ar[ru]&&\udots\ar[ru]&&\udots\ar[ru]&&\udots\ar[ru]&&\ar[ru]\\
    &\dots\dots&&\udots\ar@{-->}[ru]&&\udots\ar@{-->}[ru]&&\udots\ar@{-->}[ru]&&\udots\ar@{-->}[ru]&&\dots\dots\\
    && (-q,2q+1)\ar[ru]\ar[rd]&& (0,3q+1)\ar[ru]\ar[rd]&& (q,4q+1)\ar[ru]\ar[rd]&& (2q, 5q+1)\ar[ru]\ar[rd]&&\\
    &(-q, q+1)\ar[ru]\ar[rd]&& (0,2q+1)\ar[ru]\ar[rd]&& (q,3q+1)\ar[ru]\ar[rd]&& (2q,4q+1)\ar[ru]\ar[rd]&&\bullet\\
    (-q,1)\ar[ru]&& (0,q+1)\ar[ru]&& (q,2q+1)\ar[ru]&& (2q,3q+1)\ar[ru]&& \bullet\ar[ru]
    }}}
    \caption{Coordinate system on $\mathbb{Z} A_p$.}
    \label{fig:coordinate_ZAn}
\end{figure}

\begin{defn}[{\cite[Remark 2.3]{MJ}}]\label{defn_sigma_tau}
Define the following automorphisms on $\mathbb{Z}A_p$:
\begin{align*}
    \Sigma&: \mathbb{Z}A_p\rightarrow \mathbb{Z}A_p, &&(i,j)\mapsto (j-1,i+(p+1)q+1),\\ 
   \tau&:  \mathbb{Z}A_p\rightarrow \mathbb{Z}A_p, &&(i,j)\mapsto (i-q,j-q),
\end{align*}
and let $\tau_{q+1}=\tau\circ\Sigma^{-q}$.
\end{defn}

Note that $(\mathbb{Z} A_p,\tau)$ is a translation quiver in the sense of \cite[Definition 2.2]{MJ}. Hence there exists a \textit{mesh category} associated to it. The objects of this category are the vertices of $\mathbb{Z} A_p$ and the morphisms are linear combinations of compositions of the arrows of $\mathbb{Z} A_p$ subject to the \textit{mesh relations}. For each arrow $\alpha:x\rightarrow y$, let $\sigma (\alpha)$ be the unique arrow $\sigma (\alpha):\tau(y)\rightarrow x$. The mesh relations are given by
\begin{align*}
    \sum_{\alpha:x\rightarrow y} \alpha\sigma(\alpha)=0,
\end{align*}
for each vertex $y$ in $\mathbb{Z}A_p$.

\begin{defn}\label{defn_H+_H-}
Let $a=(rq,sq+1)$ be a vertex in $\mathbb{Z}A_p$, for $r$ and $s$ integers such that $r+1\leq s\leq r+p$. We denote by
\begin{itemize}
    \item $H^+(a)$ the set of vertices of $\mathbb{Z}A_p$ of the form $(iq, jq+1)$ for integers $r\leq i\leq s-1$ and $s\leq j\leq r+p$;
    \item $H^-(a)$  the set of vertices of $\mathbb{Z}A_p$ of the form $(iq, jq+1)$ for integers $s-p\leq i\leq r$ and $r+1\leq j\leq s$.
\end{itemize}
Note that the sets of vertices $H^-(a)$ and $H^+(a)$ are those in the ``hammocks'' spanned from $a$, see Figure \ref{fig:morphisms}.
\end{defn}

\begin{figure}
  \centering
    \begin{tikzpicture}[scale=1.8][H] 
\draw [dotted] (-5,3)--(3,3);
\draw [dotted] (-5,0)--(3,0);

\draw (0,0)--(2,2);
\draw (-1,1)--(1,3);
\draw (-1,1)--(0,0);
\draw (2,2)--(1,3);
\fill[blue!20!white] (0,0) -- (2,2) -- (1,3) -- (-1,1) -- (0,0);

\draw (-2,0)--(-1,1);
\draw (-1,1)--(-3,3);
\draw (-2,0)--(-4,2);
\draw (-2,2)--(0,0);
\draw (-3,3)--(-4,2);
\fill[green!20!white] (-1,1) -- (-3,3) -- (-4,2) -- (-2,0) -- (-1,1);
  
\draw (-1,1)  node{$\scriptstyle \bullet$}; 
\draw (-1,1.1)  node{$\scriptstyle a$};
 

\draw (0,0)  node{$\scriptstyle \bullet$};
\draw (0,-0.2)  node{$\scriptstyle ((s-1)q,\,sq+1)$};

\draw (2,2)  node{$\scriptstyle \bullet$};
\draw (2.8,2)  node{$\scriptstyle ((s-1)q,\,(r+p)q+1)$};

\draw (1,3)  node{$\scriptstyle \bullet$};
\draw (1,3.2)  node{$\scriptstyle (rq,\,(r+p)q+1)$};

\draw (-3,3)  node{$\scriptstyle \bullet$};
\draw (-3,3.2)  node{$\scriptstyle ((s-p)q,\,sq+1)$};

\draw (-4,2)  node{$\scriptstyle \bullet$};
\draw (-4.8,2)  node{$\scriptstyle ((s-p)q,\,(r+1)q+1)$};

\draw (0.6,1.6) node{$\scriptstyle H^+(a)$};
\draw (-2.6,1.6) node{$\scriptstyle H^-(a)$};


\draw (-2,0)  node{$\scriptstyle \bullet$};
\draw (-2,-0.2)  node{$\scriptstyle (rq,\,(r+1)q+1)$};


\end{tikzpicture}
 \caption{The regions $H^+(a)$ and $H^-(a)$ for $a=(rq,sq+1)$.}
    \label{fig:morphisms}
\end{figure}
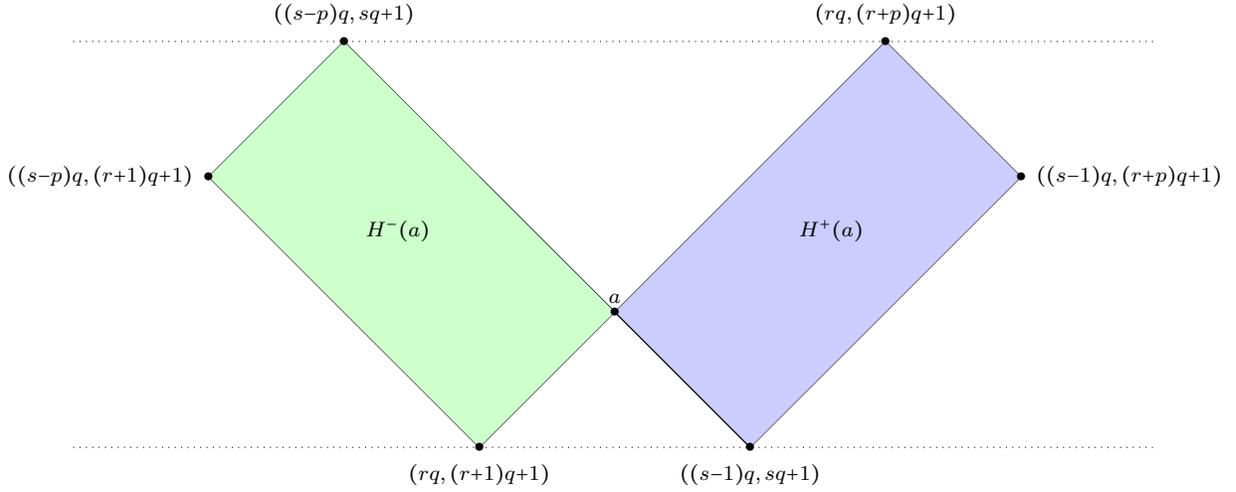

\begin{remark}\label{remark_H+_H-}
By \cite[remark 2.3]{MJ}, the regions $H^-(a)$ and $H^+(a)$ describe the set of vertices from which (respectively, to which) the object $a$ has a non-zero morphism in the mesh category associated to $\mathbb{Z}A_p$.
\end{remark}

\begin{remark}[{\cite[Section 2]{MJ}}]
The Auslander--Reiten quiver of the bounded derived category of finite--dimensional modules $\mathcal{D}^b(k A_p)$ is isomorphic, as a stable translation quiver, to $\mathbb{Z}A_p$.
The automorphisms $\Sigma$ and $\tau$ from Definition \ref{defn_sigma_tau} are the action of the suspension and the Auslander--Reiten translation in $\mathcal{D}^b(k A_p)$ respectively, expressed in terms of the coordinate system from Figure \ref{fig:coordinate_ZAn}. Moreover, the mesh category $k(\mathbb{Z}A_p)$ is equivalent to $\text{Ind}\,\mathcal{D}^b(k A_p)$, \textbf{i.e.} the full subcategory of $\mathcal{D}^b(k A_p)$ whose objects are the indecomposable objects.

The quotient translation quiver $\mathbb{Z}A_p/\langle \tau_{q+1} \rangle$ is obtained by identifying the vertices and arrows of $\mathbb{Z}A_p$ with their $\tau_{q+1}$-shifts. It is the Auslander--Reiten quiver of 
\begin{align*}
    \mathcal{C}_q(A_p):=\mathcal{D}^b(k A_p)/\tau\circ\Sigma^{-q},
\end{align*}
the \textit{triangulated $q$-cluster category of Dynkin type $A_p$}.
Figure \ref{fig:quotient_tau} shows the identification on $\mathbb{Z}A_p$ when $q$ is odd. Note that in this case, the quiver can be drawn on a M\"{o}bius strip.

Moreover, note that $\mathcal{C}_q(A_p)$ is a $(q+1)$-Calabi--Yau category whose $\Hom$-spaces between indecomposables are either zero or one dimensional over $k$.
\end{remark}
\begin{figure}
\centering
\begin{tikzpicture}[scale=2][H]
\draw[red](1,1.6)--(0,0);
\draw(0,0)--(2,0)--(1,1.6);
\draw(1.2,1.6)--(2.2,0)--(3.2,1.6)--(1.2,1.6);
\draw[very thick, dotted](3.4,0.8)--(3.7,0.8);
\draw (4,0)--(6,0)--(5,1.6)--(4,0);
\draw (5.2,1.6)--(6.2,0);
\draw[red](5.4,1.6)--(6.4,0);

\draw (0,0)  node{$\scriptstyle \bullet$};
\draw (-0.35,0)  node{$\scriptstyle (0,q+1)$};

\draw (2,0)  node{$\scriptstyle \bullet$};
\draw (2.2,0)  node{$\scriptstyle \bullet$};
\draw (4,0)  node{$\scriptstyle \bullet$};    
\draw (6,0)  node{$\scriptstyle \bullet$};
\draw (6.2,0)  node{$\scriptstyle \bullet$};

\draw (6.4,0)  node{$\scriptstyle \bullet$};
\draw (7.35,0)  node{$\scriptstyle \tau^{-1}_{q+1}(0,pq+1)\equiv (0,pq+1)$};

\draw (1,1.6)  node{$\scriptstyle \bullet$};
\draw (0.6,1.6)  node{$\scriptstyle (0,pq+1)$};

\draw (1.2,1.6)  node{$\scriptstyle \bullet$};
\draw (3.2,1.6)  node{$\scriptstyle \bullet$};
\draw (5,1.6)  node{$\scriptstyle \bullet$};
\draw (5.2,1.6)  node{$\scriptstyle \bullet$};

\draw (5.4,1.6)  node{$\scriptstyle \bullet$};
\draw (6.25,1.6)  node{$\scriptstyle \tau^{-1}_{q+1}(0,q+1)\equiv (0,q+1)$};

\draw (1,0.8) node{$\scriptstyle 1$};
\draw (2.2,0.8) node{$\scriptstyle 2$};
\draw (5,0.8) node{$\scriptstyle q$};
\end{tikzpicture}
\caption{The quotient translation quiver $\mathbb{Z}A_p/\langle \tau_{q+1} \rangle$ when $q$ is odd.}
\label{fig:quotient_tau}
\end{figure}
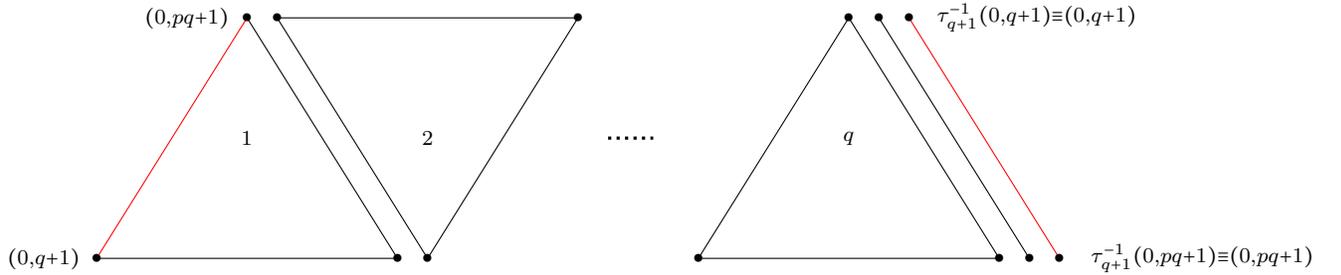

We present a geometric realisation of $\mathbb{Z}A_p/\langle \tau_{q+1} \rangle$. Let $N=(p+1)q+2$ and $P$ be a regular convex $N$-gon. Label the vertices of $P$ from $0$ to $N-1$ in an anti-clockwise direction.
We denote the diagonal joining vertices $i$ and $j$ by $\{i,j\}$.

\begin{defn}[{\cite[Definition 2.5]{MJ}}]
A \textit{$q$-allowable diagonal in $P$} is a diagonal joining two non-adjacent boundary vertices which divides $P$ into two smaller polygons which can themselves be subdivided into $(q+2)$-gons by non-crossing diagonals. Note that these are the diagonals of $P$ spanning $1+kq$ vertices, for $k$ a positive integer.
\end{defn}

\begin{proposition}[{\cite[proposition 2.9]{MJ}}]
There is a bijection
\begin{align*}
    \Bigg\{
    \begin{matrix}\text{indecomposable objects in } \mathcal{C}_q(A_p)  \\ \,\, (=\text{vertices of } \mathbb{Z}A_p/\langle \tau_{q+1} \rangle) \,\, \end{matrix}
    \Bigg\}
    \longrightarrow
    \Bigg\{
    \begin{matrix} \,\,\text{$q$-allowable diagonals}\,\, \\ \text{in } P \end{matrix}
    \Bigg\}
\end{align*}
given by $(i,j)\mapsto\{i\, (\mmod N),j\, (\mmod N)\}$.
\end{proposition}
From now on, $q$-allowable diagonals in $P$ and indecomposable objects in $\mathcal{C}_q(A_p)$ are identified. Hence it makes sense to talk about morphisms between two $q$-allowable diagonals.

\begin{defn}
A \textit{$(q+2)$-angulation of $P$} is a maximal collection of non-crossing $q$-allowable diagonals.
\end{defn}

\begin{notation}
Given a vertex $v$ of $P$ and an integer $r$, we denote by $v^r$ its $r^{th}$ successor in the anti-clockwise direction if $r$ is positive and  its $(-r)^{th}$ successor in the clockwise direction if $r$ is negative. We also use the convention $v^{0}=v$.
\end{notation}

\begin{remark}
We can define a cyclic order on the vertices of $P$ as follows. Given three vertices $u,v,w$ of $P$, we write $u<v<w$ if they appear in the order $u,v,w$ when going through the vertices of $P$ in the anti-clockwise direction. Moreover, if we choose two distinct vertices $u$ and $v$, we can consider the interval of vertices $[u,v]$ and in this ``$<$'' is a total order.
\end{remark}

The next two lemmas follow by describing the $H^+(b)$, respectively $H^-(b)$, region from Definition \ref{defn_H+_H-} in terms of $q$-allowable diagonals in $P$.
\begin{lemma}\label{lemma_morphisms_from}
Consider a $q$-allowable diagonal $b=\{b_0,b_1\}$ in $P$. Then a $q$-allowable diagonal $a$ is such that $\Hom (b,a)\neq 0$ if and only if there are some non-negative integers $i,\, j$ such that $a=\{a_0,a_1\}$ for
\begin{align*}
    a_0=b_0^{iq}\in [b_0, b_1^{-2}] \text{ and } a_1=b_1^{jq}\in [b_1, b_0^{-2}].
\end{align*}
See Figure \ref{fig:fig_morphisms_from_b}.
\end{lemma}

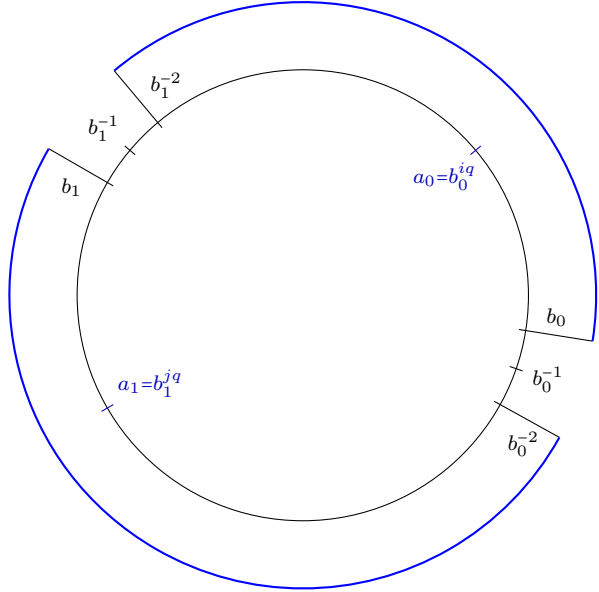
\begin{figure}
\centering
\begin{tikzpicture}[scale=3]
   \draw (0,0) circle (1cm);
   
   \draw (-9:0.97cm) -- (-9:1.3cm);
   \draw (-5:1.13cm) node{$\scriptstyle b_{0}$};
   \draw (150:0.97cm) -- (150:1.3cm);
   \draw (155:1.13cm) node{$\scriptstyle b_{1}$};
   
   \draw (-19:0.97cm) -- (-19:1.03);
   \draw (-19:1.15cm) node{$\scriptstyle b_{0}^{-1}$}; 
   \draw (140:0.97cm) -- (140:1.03cm); 
   \draw (140:1.15cm) node{$\scriptstyle b_{1}^{-1}$};
   \draw (-29:0.97cm) -- (-29:1.3);
   \draw (-34:1.18cm) node{$\scriptstyle b_{0}^{-2}$};
   \draw (130:0.97cm) -- (130:1.3cm); 
   \draw (123:1.11cm) node{$\scriptstyle b_{1}^{-2}$};
   
   \draw [blue!75!black](40:0.97cm) -- (40:1.03cm);
   \draw [color=blue!75!black] (40:0.82cm) node{$\scriptstyle a_0=b_{0}^{iq}$};
   
   \draw [blue!75!black](-150:0.97cm) -- (-150:1.03cm);
   \draw [color=blue!75!black] (-150:0.78cm) node{$\scriptstyle a_{1}=b_{1}^{jq}$};

  
   \draw[thick, blue] ([shift=(130:1.3cm)]0,0) arc (130:-9:1.3cm);
   \draw[thick, blue] ([shift=(-210:1.3cm)]0,0) arc (-210:-29:1.3cm);
\end{tikzpicture} 
\caption{There is a non-zero morphism $b=\{ b_0,b_1 \}\rightarrow \{ a_0,a_1\}=a$.}
\label{fig:fig_morphisms_from_b}
\end{figure}

\begin{lemma}\label{lemma_morphisms_to}
Consider a $q$-allowable diagonal $b=\{b_0,b_1\}$ in $P$. Then a $q$-allowable diagonal $a$ is such that $\Hom (a,b)\neq 0$ if and only if there are some non-negative integers $i,\, j$ such that $a=\{a_0,a_1\}$ for
\begin{align*}
    a_0=b_0^{-iq}\in [b_1^{2},b_0] \text{ and } a_1=b_1^{-jq}\in [b_0^{2}, b_1].
\end{align*}
See Figure \ref{fig:fig_morphisms_to_b}.
\end{lemma}
\begin{figure}
\centering
\begin{tikzpicture}[scale=3]
   \draw (0,0) circle (1cm);
   
   \draw (-9:0.97cm) -- (-9:1.3cm);
   \draw (-13:1.13cm) node{$\scriptstyle b_{0}$};
   \draw (150:0.97cm) -- (150:1.3cm);
   \draw (146:1.13cm) node{$\scriptstyle b_{1}$};
   
   \draw (1:0.97cm) -- (1:1.03);
   \draw (1:1.15cm) node{$\scriptstyle b_{0}^{1}$}; 
   \draw (160:0.97cm) -- (160:1.03cm); 
   \draw (160:1.15cm) node{$\scriptstyle b_{1}^{1}$};
   \draw (11:0.97cm) -- (11:1.3);
   \draw (16:1.15cm) node{$\scriptstyle b_{0}^{2}$};
   \draw (170:0.97cm) -- (170:1.3cm); 
   \draw (176:1.11cm) node{$\scriptstyle b_{1}^{2}$};
   
   \draw [green](80:0.97cm) -- (80:1.03cm);
   \draw [color=green] (78:0.86cm) node{$\scriptstyle a_{1}=b_1^{-jq}$};
   \draw [green](-150:0.97cm) -- (-150:1.03cm);
   \draw [color=green] (-152:0.80cm) node{$\scriptstyle a_{0}=b_0^{-iq}$};

  
   \draw[thick, green] ([shift=(150:1.3cm)]0,0) arc (150:11:1.3cm);
   \draw[thick, green] ([shift=(170:1.3cm)]0,0) arc (170:351:1.3cm);
   
\end{tikzpicture} 
\caption{There is a non-zero morphism $a=\{a_0,a_1\}\rightarrow\{b_0,b_1\}=b$.}
\label{fig:fig_morphisms_to_b}
\end{figure}
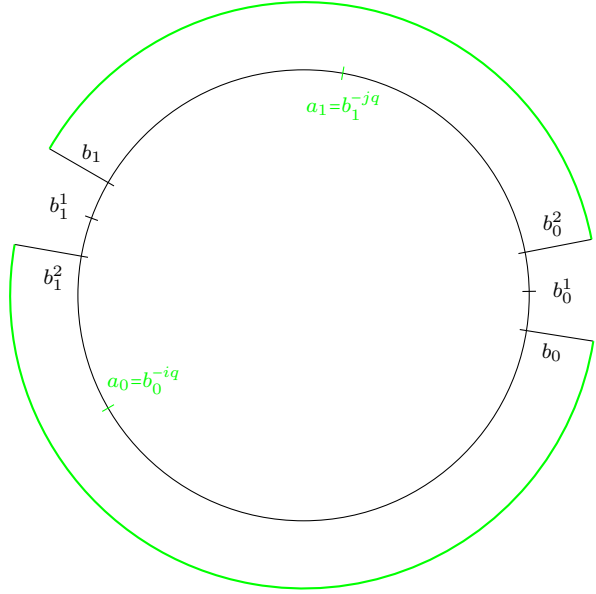

The following is a consequence of \cite[Proposition 2.9]{MJ}.

\begin{lemma}\label{lemma_crossing_Ext}
Two $q$-allowable diagonals $a$ and $b$ in $P$ cross if and only if there exists an integer $1\leq i\leq q$ such that $\Ext^i(a,b)\neq 0$.
\end{lemma}

We describe all the triangles in $\mathcal{C}_q(A_p)$ with indecomposable end terms in terms of $q$-allowable diagonals in $P$.
In order to do this, we use a similar method to the one used by Pescod in \cite[chapter 4]{PD}.

The following two lemmas are inspired by \cite[lemmas 4.1.1 and 4.1.2]{PD}.
\begin{lemma}\label{lemma_PD411}
Consider a triangle in $\mathcal{C}_q(A_p)$ of the form
\begin{align*}
\Delta=a\rightarrow e\rightarrow b\rightarrow \Sigma a,
\end{align*}
with $a$ and $b$ indecomposable. If $c$ is an indecomposable in $\mathcal{C}_q(A_p)$ such that there exists an integer $1\leq i\leq q$ with $\Ext^i(c,e)\neq 0$, then at least one of $\Ext^i(c,a)$ and $\Ext^i(c,b)$ is non-zero.

In terms of $q$-allowable diagonals in $P$, we have that if $c$ crosses a direct summand of $e$, then $c$ crosses at least one of $a$ and $b$.
\end{lemma}

\begin{proof}
The triangle $\Delta$ induces the exact sequence:
\begin{align*}
\Ext^i(c,a)\rightarrow \Ext^i(c,e)\rightarrow \Ext^i(c,b).
\end{align*}
Since $\Ext^i(c,e)\neq 0$, it follows that at least one of $\Ext^i(c,a)$ and $\Ext^i(c,b)$ is non-zero.
%
\end{proof}

\begin{lemma}\label{lemma_PD412}
Let $a$ and $b\in\mathcal{C}_q(A_p)$ be indecomposable and assume that $\Ext^1(b,a)\neq 0$. Let
\begin{align*}
\Delta= a\xrightarrow{\alpha} e\xrightarrow{\epsilon} b\xrightarrow{\beta} \Sigma a
\end{align*}
be the ensuing nonsplit triangle. Then
\begin{align*}
\Ext^{1\dots q}(e,a)=\Ext^{1\dots q}(b,e)=0.
\end{align*}
In terms of $q$-allowable diagonals in $P$, there is no direct summand of $e$ crossing $a$ or $b$.
\end{lemma}

\begin{proof}
Note that the identification on $\mathbb{Z}A_p$ to obtain $\mathbb{Z}A_p/\langle \tau_{q+1}\rangle \cong \mathcal{C}_q(A_p)$ is such that
\begin{align*}
H^+(\Sigma^{-q}b),\, H^+(\Sigma^{-q+1}b),\, \dots,\, H^+(\Sigma^{-2}b),\, H^+(\Sigma^{-1}b)
\end{align*}
are all disjoint. See Figure \ref{fig:quotient_tau} for $\mathbb{Z}A_p/\langle \tau_{q+1}\rangle$ when $q$ is odd; the case when $q$ is even is similar.
By Remark \ref{remark_H+_H-},  we have that at most one of
\begin{align*}
\Hom(\Sigma^{-q}b, a),\, \Hom(\Sigma^{-q+1}b,a),\,\dots,\, \Hom(\Sigma^{-2}b,a),\, \Hom(\Sigma^{-1}b,a)
\end{align*}
is non-zero. Equivalently, at most one of
\begin{align*}
\Ext^q(b,a),\, \Ext^{q-1}(b,a),\,\dots,\, \Ext^{2}(b,a),\, \Ext^{1}(b,a)
\end{align*}
is non-zero. Since $\Ext^1(b,a)$ is non-zero by assumption, we have that $\Ext^{2\dots q}(b,a)=0$.
Consider the following exact sequence induced by $\Delta$:
\begin{align*}
&\Hom (b,b)\xrightarrow{\beta_*} \Hom(b,\Sigma a)\xrightarrow{\gamma} \Hom(b, \Sigma e)\rightarrow\Hom(b,\Sigma b).
\end{align*}
Since $b$ does not cross itself, by Lemma \ref{lemma_crossing_Ext}, we have that $\Hom(b,\Sigma^i b)=0$ for any $1\leq i\leq q$. Then, as $\Hom(b,b)$ is non-zero and $\beta\neq 0$, we have that $\beta_*\neq 0$. Hence $\Hom(b,\Sigma a)$ is one-dimensional over $k$ and $\beta_*$ is surjective, so that $\gamma=0$. Then $\Ext^1(b, e)=\Hom(b,\Sigma e)=0$. For $2\leq i\leq q$, consider the following exact sequence induced by $\Delta:$
\begin{align*}
\Hom(b, \Sigma^i a)\rightarrow \Hom(b, \Sigma^i e)\rightarrow \Hom(b, \Sigma^i b)=0.
\end{align*} 
For $2\leq i\leq q$, we have 
\begin{align*}
0=\Ext^i(b,a)=\Hom(b, \Sigma^i a),
\end{align*}
and so $\Ext^{2\dots q}(b,e)=0$. 

A similar argument shows that no direct summand of $e$ crosses $a$.
\end{proof}

Note that the results from \cite[chapter 4.2]{PD} are valid in $\mathcal{C}_q(A_p)$ since this is a $k$-linear, $\Hom$-finite, Krull-Schmidt triangulated category. We state them again for the convenience of the reader.

\begin{lemma}[{\cite[lemmas 4.2.1 and 4.2.2]{PD}}]
Let $b\in\mathcal{C}_q(A_p)$ be indecomposable. Assume there exists a nonsplit triangle
\begin{align*}
a\xrightarrow{\alpha} e\xrightarrow{\epsilon} b\rightarrow \Sigma a,
\end{align*}
then each row of the matrix $\alpha$ has a non-zero entry and each column of the matrix $\epsilon$ has a non-zero entry.
\end{lemma}

\begin{remark}
Note that since $\Hom$-spaces between indecomposables in $\mathcal{C}_q(A_p)$ are either zero or one dimensional over $k$, we can state  \cite[lemma 4.2.3]{PD} as follows.
\end{remark}

\begin{lemma}\label{lemma_PD423}
For $a$ and $b$ indecomposables in $\mathcal{C}_q(A_p)$, let
\begin{align*}
a\rightarrow e\rightarrow b\rightarrow \Sigma a
\end{align*}
be a triangle in $\mathcal{C}_q(A_p)$. Then, $e$ has no repeated indecomposable summands.
\end{lemma}

\begin{lemma}[{\cite[lemma 4.2.4]{PD}}]\label{lemma_PD424}
For $a$ and $b$ indecomposables in $\mathcal{C}_q(A_p)$, let
\begin{align*}
a\rightarrow e\rightarrow b\rightarrow \Sigma a
\end{align*}
be a triangle in $\mathcal{C}_q(A_p)$. Then
\begin{align*}
\Ext^{1}(\Sigma a, e_i)\neq 0 \text{ and } \Ext^{1}(e_i,\Sigma^{-1}b)\neq 0,
\end{align*}
for each indecomposable summand $e_i$ of $e$.
\end{lemma}

\begin{lemma}\label{lemma_allowable_iff}
Consider two crossing $q$-allowable diagonals $a=\{a_0,a_1\}$ and $b=\{b_0,b_1\}$ in $P$, where $b_0<a_0<b_1<a_1$. Then $\{a_0,b_0\}$ is $q$-allowable or an edge if and only if $\{a_1,b_1\}$ is $q$-allowable or an edge.
\end{lemma}
\begin{proof}
Since $a$ and $b$ are $q$-allowable diagonals, there are positive integers $r,\, t$ such that
\begin{align*}
    a_1=a_0^{1+rq} \text{ and } b_1=b_0^{1+tq}.
\end{align*}
Assume that $(a_0, b_0)$ is $q$-allowable or an edge in $P$. Then $a_0=b_0^{1+sq}$ for some integer $s\geq 0$, see Figure \ref{fig:fig_allowable}. We have 
\begin{align*}
    a_1= a_0^{1+rq}=b_0^{2+(r+s)q}=b_1^{2+(r+s)q-1-tq}=b_1^{1+(r+s-t)q},
\end{align*}
where $r+s\geq t$. Hence $(a_1,b_1)$ is $q$-allowable or an edge in $P$. The other direction is proved in a similar way.
\end{proof}

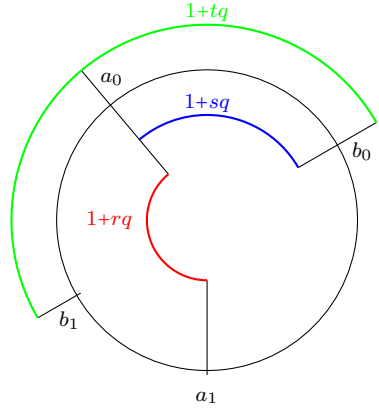
\begin{figure}
  \centering
    \begin{tikzpicture}[scale=2]
      \draw (0,0) circle (1cm);

      \draw (30:0.7cm) -- (30:1.3cm);
      \draw (24:1.13cm) node{$\scriptstyle b_{0}$};
      
      \draw (210:0.97cm) -- (210:1.3cm);
      \draw (216:1.13cm) node{$\scriptstyle b_{1}$};

      \draw (-90:0.4cm) -- (-90:1.03);
      \draw (-90:1.18cm) node{$\scriptstyle a_{1}$};

      \draw (130:0.4cm) -- (130:1.3cm); 
      \draw (124:1.13cm) node{$\scriptstyle a_{0}$};
      
    \draw[thick, green] ([shift=(30:1.3cm)]0,0) arc (30:210:1.3cm);
    \draw[thick, blue] ([shift=(30:0.7cm)]0,0) arc (30:130:0.7cm);
    \draw[thick, red] ([shift=(130:0.4cm)]0,0) arc (130:270:0.4cm);
    
    \draw [green](90:1.38cm) node{$\scriptstyle 1+tq$};
    \draw [blue](90:0.78cm) node{$\scriptstyle 1+sq$};
    \draw [red](180:0.65cm) node{$\scriptstyle 1+rq$};
      
    \end{tikzpicture}
    \caption{The diagonals $\{a_0,a_1\}$ and $\{b_0,b_1\}$ are $q$-allowable and $\{a_0,b_0\}$ is either $q$-allowable or an edge.}
    \label{fig:fig_allowable}
\end{figure}

\begin{proposition}\label{prop_triangles_Pescod}
Consider two crossing $q$-allowable diagonals $a=\{a_0,a_1\}$ and $b=\{b_0,b_1\}$ in $P$, where $b_0<a_0<b_1<a_1$.
\begin{itemize}
\item There exists exactly one integer $0\leq l\leq q-1$ such that $\Hom(b,\Sigma^{l+1}a)\neq 0$. Then the nonsplit triangle extending $\beta: b\rightarrow \Sigma^{l+1}a$ is
\begin{align*}
\Delta= \Sigma^l a\rightarrow e\rightarrow b\xrightarrow{\beta} \Sigma^{l+1}a,
\end{align*}
where $e=e_1\oplus e_2$ for $e_{1}=\{a_0^{-l},b_0\}$ and $e_2=\{b_1, a_1^{-l}\}$.
\item If $0\leq i \leq q-1$ is an integer such that $\{a_0^{-i},b_0\}$ is a $q$-allowable diagonal or an edge in $P$, then $i=l$.
\end{itemize}
\end{proposition}

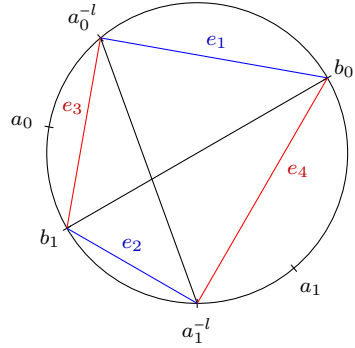
\begin{figure}
  \centering
    \begin{tikzpicture}[scale=2]
      \draw (0,0) circle (1cm);

      \draw (30:0.97cm) -- (30:1.03cm);
      \draw (30:1.13cm) node{$\scriptstyle b_{0}$};
      
      \draw (210:0.97cm) -- (210:1.03cm);
      \draw (210:1.13cm) node{$\scriptstyle b_{1}$};
      
      \draw (-50:0.97cm) -- (-50:1.03);
      \draw (-50:1.18cm) node{$\scriptstyle a_{1}$};
      
       \draw (170:0.97cm) -- (170:1.03cm); 
      \draw (170:1.18cm) node{$\scriptstyle a_{0}$};
     
      \draw (-90:0.97cm) -- (-90:1.03);
      \draw (-90:1.18cm) node{$\scriptstyle a_{1}^{-l}$};

      \draw (130:0.97cm) -- (130:1.03cm); 
      \draw (130:1.18cm) node{$\scriptstyle a_{0}^{-l}$};
      
      \draw (210:1cm) -- (30:1cm);
      \draw (-90:1cm) -- (130:1cm);
     
     \draw [blue] (-90:1cm) -- (210:1cm);
     \draw [blue] (30:1cm) -- (130:1cm);
     
     \draw [red] (130:1cm) -- (210:1cm);
     \draw [red] (30:1cm) -- (-90:1cm);
     
     \draw [color=blue!75!black] (80:0.75cm) node{$\scriptstyle e_1$};
     \draw [color=blue!75!black] (-125:0.75cm) node{$\scriptstyle e_2$};
     \draw [color=red!75!black] (160:0.88cm) node{$\scriptstyle e_3$};
      \draw [color=red!75!black] (-10:0.68cm) node{$\scriptstyle e_4$};
    \end{tikzpicture}
    \caption{The triangle $\Sigma^l a\rightarrow e_1\oplus e_2\rightarrow b\xrightarrow{\beta} \Sigma^{l+1}a$.}
    \label{fig:figtriangle}
\end{figure}

\begin{proof}
Assume that $\Delta$ is a non-split triangle and let $\overline{e}$ be a direct summand of the middle term $e$. By Lemma \ref{lemma_PD412}, we have that $\overline{e}$ does not cross $\Sigma^l a$ or $b$. Moreover, if $c$ is a $q$-allowable diagonal crossing $\overline{e}$, by Lemma \ref{lemma_PD411}, we have that $c$ crosses at least one of $\Sigma^l a$ and $b$. Hence the only possibilities for $\overline{e}$ are the diagonals
\begin{align*}
    e_{1}=\{a_0^{-l},b_0\},\, e_2=\{b_1, a_1^{-l}\},\, e_3=\{a_0^{-l},b_1\},\, e_4=\{b_0, a_1^{-l}\},
\end{align*}
see Figure \ref{fig:figtriangle}. By Lemma \ref{lemma_PD424}, we have  that
\begin{align*}
    \Ext^{1}(\Sigma^{l+1} a, \overline{e})\neq 0 \text{ and } \Ext^{1}(\overline{e},\Sigma^{-1}b)\neq 0.
\end{align*}
Hence $\overline{e}$ crosses both $\Sigma^{l+1}a$ and $\Sigma^{-1} b$, which implies that $e_3$ and $e_4$ must be excluded from the possible summands of $e$. Moreover, $e$ has no repeated summands by Lemma \ref{lemma_PD423}, and so
\begin{align*}
    e\in\{ 0, e_1, e_2, e_1\oplus e_2 \}.
\end{align*}
We claim that $e=e_1\oplus e_2$. We prove this claim by dealing with the cases $e_1=e_2=0$, one of $e_1$, $e_2$ zero and $e_1$, $e_2$ both non-zero separately. 
First, note that if $b=\Sigma^{l+1}a$, then
\begin{align*}
    \Delta= \Sigma^{l}a\rightarrow 0\rightarrow \Sigma^{l+1} a\xrightarrow{\cong} \Sigma^{l+1} a.
\end{align*}
Note that in this case $b_0^{1}=a_0^{-l}$ and $b_1^{1}=a_1^{-l}$ so that $e_1=e_2=0$ and $e=e_1\oplus e_2=0$.
Assume now that $b$ is not $\Sigma^{l+1}a$, so that $e\neq 0$.
Note that if $e_1$ (respectively $e_2$) is zero, then $e=e_1\oplus e_2=e_2$ (respectively $e=e_1$) is the only option and we are done. Moreover, since $e\neq 0$, we have that at least one of $e_1$, $e_2$ is $q$-allowable or an edge in $P$. But then, by Lemma \ref{lemma_allowable_iff}, we have that $e_1$ and $e_2$ are both $q$-allowable diagonals or edges in $P$.
The last case to deal with is when $e_1$ and $e_2$ are both non-zero, \textbf{i.e.} they both are $q$-allowable diagonals in $P$.
Suppose for a contradiction that $e=e_1$.
Consider the triangle
\begin{align*}
    e\rightarrow b\rightarrow \Sigma^{l+1}a\rightarrow \Sigma e,
\end{align*}
and set $c=\{a_0^{-l}, b_1^{1}\}$. Note that $b_1<b_1^{1}< a_1^{-l}$ since $e_2\neq 0$. Since $\Sigma^l a$ and $e_2$ are $q$-allowable diagonals, there are integers $s>r>0$ such that
\begin{align*}
    a_1^{-l}=a_0^{-l+1+sq}=b_1^{1+rq}.
\end{align*}
Hence $b^{1}_1=a_0^{-l+1+(s-r)q}$ and so $c$ is $q$-allowable. So $c$ is a $q$-allowable diagonal crossing $b$ but not crossing neither $e=e_1$ nor $\Sigma^{l+1}a$, see Figure \ref{fig:figtriangleproof}, contracting Lemma \ref{lemma_PD411}. Then $e\neq e_1$.
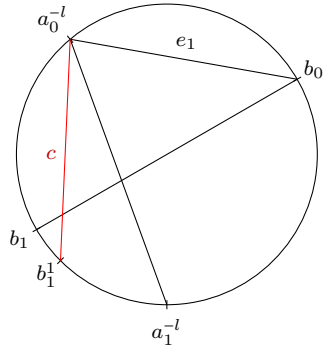
\begin{figure}
  \centering
    \begin{tikzpicture}[scale=2]
      \draw (0,0) circle (1cm);

      \draw (30:0.97cm) -- (30:1.03cm);
      \draw (30:1.13cm) node{$\scriptstyle b_{0}$};
      
      \draw (210:0.97cm) -- (210:1.03cm);
      \draw (210:1.13cm) node{$\scriptstyle b_{1}$};
      
     \draw (225:0.97cm) -- (225:1.03cm);
     \draw (225:1.13cm) node{$\scriptstyle b_{1}^1$}; 
      
      
     
      \draw (-90:0.97cm) -- (-90:1.03);
      \draw (-90:1.18cm) node{$\scriptstyle a_{1}^{-l}$};

      \draw (130:0.97cm) -- (130:1.03cm); 
      \draw (130:1.18cm) node{$\scriptstyle a_{0}^{-l}$};
      
      \draw (210:1cm) -- (30:1cm);
      \draw (-90:1cm) -- (130:1cm);
     
     \draw  (30:1cm) -- (130:1cm);
     
     \draw [red] (225:1cm) -- (130:1cm);
     
     
     \draw (80:0.75cm) node{$\scriptstyle e_1$};
     \draw [color=red!75!black] (180:0.77cm) node{$\scriptstyle c$};
    \end{tikzpicture}
    \caption{The $q$-allowable diagonal $c$ crosses $b$ but does not cross neither $e_1$ nor $\Sigma^{l+1}a$.}
    \label{fig:figtriangleproof}
\end{figure}
By a similar argument, assuming that $e=e_2$ leads to a contradiction.
Hence we have that $e=e_1\oplus e_2$.

We now prove the second part of the proposition. Assume $0\leq i\leq q-1$ is an integer such that $\{a_0^{-i},b_0\}$ is a $q$-allowable diagonal or an edge in $P$. Then there are integers $r,\,s \geq 0$ such that
\begin{align*}
    a_0^{-i}=b_0^{1+sq} \text{ and } a_0^{-l}=b_0^{1+rq}.
\end{align*}
Then $b_0^{1+rq}=a_0^{-l}=b_0^{1+sq+i-l}$. So $i-l=(r-s)q$.

If $0\leq l\leq i\leq q-1$, then $r\geq s\geq 0$. Note that $0\leq i-l\leq q-l-1<q$, so $i=l$ is the only option. If $0\leq i\leq l\leq q-1$, then $s\geq r\geq 0$. Note that $0\leq l-i\leq q-i-1<q$, so $i=l$ is again the only option.
\end{proof}

\begin{proposition}\label{prop_Grot_Cm}
Assume that $q$ is odd.
We have that
\begin{align*}
    K_0(\mathcal{C}_q(A_p))\cong
    \begin{cases}
    0 &\text{if $p$ is even,}\\
    \mathbb{Z} & \text{if $p$ is odd.}
    \end{cases}
\end{align*}
\end{proposition}

\begin{proof}
Consider the $(q+2)$-angulation $T={T_0,\,\dots,\, T_{p-1}}$, where
\begin{align*}
    T_0=\{0,q+1\} \text{ and } T_i=\{N-i, (1+i)q+1-i\}, \text{ for $1\leq i\leq p-1$},
\end{align*}
see Figure \ref{fig:fig_T_angulation}. Note that by \cite[proposition 2.14]{MJ} this corresponds to the $(q+1)$-cluster tilting object $T_0\oplus\dots\oplus T_{p-1}$. Let $\mathcal{T}:=\text{add}(T_0\oplus\dots\oplus T_{p-1})\subseteq \mathcal{C}_q(A_p)$ be the corresponding $(q+1)$-cluster tilting subcategory.
\begin{figure}
  \centering
    \begin{tikzpicture}[scale=3]
      \draw (0,0) circle (1cm);

     \draw (0:1.07cm) node{$\scriptstyle 0$};
     \draw (-10:1.12cm) node{$\scriptstyle N-1$};
     \draw (-20:1.12cm) node{$\scriptstyle N-2$};
     \draw (-60:1.12cm) node{$\scriptstyle N-p+1$};
     \draw (80:1.07cm) node{$\scriptstyle q+1$};
     \draw (120:1.17cm) node{$\scriptstyle 2q$};
     \draw (160:1.14cm) node{$\scriptstyle 3q-1$};
     \draw (-140:1.19cm) node{$\scriptstyle pq-p+2$};

     
     \draw [red] (0:1cm) -- (80:1cm);
     \draw [red] (-10:1cm) -- (120:1cm);
     \draw [red] (-20:1cm) -- (160:1cm);
     \draw [red] (-60:1cm) -- (-140:1cm);
     
     \draw [red] (-90:0.35cm) node{$\scriptstyle \vdots$};
     \draw [red] (80:0.82cm) node{$\scriptstyle T_0$};
     \draw [red] (122:0.78cm) node{$\scriptstyle T_1$};
     \draw [red] (168:0.78cm) node{$\scriptstyle T_2$};
     \draw [red] (-138:0.82cm) node{$\scriptstyle T_{p-1}$};
    \end{tikzpicture}
    \caption{The $(q+2)$-angulation $T$.}
    \label{fig:fig_T_angulation}
\end{figure}
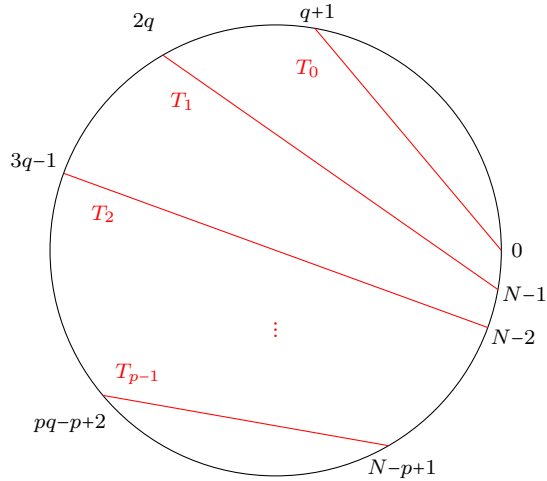
We want to find the Auslander--Reiten $(q+3)$-angle in $\mathcal{T}$ starting and ending at $T_i$ for $0\leq i\leq p-1$.


Consider $i=0$. By Lemma \ref{lemma_morphisms_to} there are no non-zero morphisms $T_i\rightarrow T_0$ for $i\neq 0$, so that $\tau_0 : 0\rightarrow T_0$ is a right almost split morphism in $\mathcal{T}$.
Consider $-q+2\leq j\leq -1$. Note that, since $\Sigma^j T_0=\{ -j, q+1-j \}$, Lemma \ref{lemma_morphisms_from} can be used to check that for all indecomposables $T_i$ in $\mathcal{T}$, we have that $\Hom(T_i,\Sigma^j T_0)=0$. Hence $\tau_{-j}: 0\rightarrow \Sigma^j T_0$ is a $\mathcal{T}$-cover for any $-q+2\leq j\leq -1$. Moreover, by Lemma \ref{lemma_morphisms_to}, we have that $\tau_{q-1}:T_1\rightarrow \Sigma^{-q+1} T_0$ is a $\mathcal{T}$-cover.
Using Proposition \ref{prop_triangles_Pescod}, the morphism $\tau_{q-1}$ extends to the non-split triangle
\begin{align*}
    \Sigma T_0\rightarrow T_1\xrightarrow{\tau_{q-1}}\Sigma^{-q+1} T_0\rightarrow \Sigma^2 T_0.
\end{align*}
Applying Lemma \ref{lemma_morphisms_to} to $\Sigma T_0=\{N-1, q\}$, we see that $\Hom(\mathcal{T}, \Sigma T_0)=0$. Hence $\tau_q: 0\rightarrow \Sigma T_0$ is a $\mathcal{T}$-cover.
The Auslander--Reiten $(q+3)$-angle starting and ending at $T_0$ is then the one corresponding to the following tower of triangles:
\begin{align*}
\begin{gathered}
\xymatrix@!0 @C=3em @R=3em{
& 0\ar[rd]^{\tau_{q}}\ar[rr] && T_{1}\ar[rd]^{\tau_{q-1}}\ar[rr]&& 0\ar[rd]^{\tau_{q-2}}\ar[r] &&\cdots&\ar[r]& 0\ar[rr]\ar[rd]^{\tau_1} &&0\ar[rd]^{\tau_0}
\\
T_0\ar[ru]&& \Sigma T_{0}\ar@{~>}[ll]\ar[ru]&& \Sigma^{-q+1} T_{0}\ar@{~>}[ll]\ar[ru]&& \Sigma^{-q+2} T_0\ar@{~>}[ll]&\cdots& \Sigma^{-2} T_{0}\ar[ru]&& \Sigma^{-1} T_{0} \ar@{~>}[ll]\ar[ru]&& T_0.\ar@{~>}[ll]
}
\end{gathered}
\end{align*}

In a similar way, we can find the remaining Auslander--Reiten $(q+3)$-angles. These are the ones corresponding to the following towers of triangles:
\begin{align*}
\begin{gathered}
\xymatrix@!0 @C=2.7em @R=2.7em{
& 0\ar[rd]^{\tau_{q}}\ar[rr] && T_{i+1}\ar[rd]^{\tau_{q-1}}\ar[rr]&&0\ar[rd]^{\tau_{q-2}}\ar[r]&&\cdots&\ar[r]& 0\ar[rr]\ar[rd]^{\tau_2} && T_{i-1}\ar[rr]\ar[rd]^{\tau_1} &&0\ar[rd]^{\tau_0}
\\
T_{i}\ar[ru]&& \Sigma T_{i}\ar@{~>}[ll]\ar[ru]&& A_i\ar@{~>}[ll]\ar[ru]&&\Sigma A_i\ar@{~>}[ll]&\cdots& \Sigma^{q-4}A_i\ar[ru]&&\Sigma^{q-3}A_i\ar@{~>}[ll]\ar[ru]&& \Sigma^{-1} T_{i} \ar@{~>}[ll]\ar[ru]&& T_{i},\ar@{~>}[ll]
}
\end{gathered}
\end{align*}
where $A_i:=\{(i+2)q-i, (i+1)q-i-1\}$, for $1\leq i\leq p-2$, and
\begin{align*}
\begin{gathered}
\xymatrix@!0 @C=3em @R=3em{
& 0\ar[rd]^{\tau_{q}}\ar[rr] && 0\ar[rd]^{\tau_{q-1}}\ar[r]&&\cdots&\ar[r]& 0\ar[rr]\ar[rd]^{\tau_2} && T_{p-2}\ar[rr]\ar[rd]^{\tau_1} &&0\ar[rd]^{\tau_0}
\\
T_{p-1}\ar[ru]&& \Sigma T_{p-1}\ar@{~>}[ll]\ar[ru]&& \Sigma^{2} T_{p-1}\ar@{~>}[ll]&\cdots& \Sigma^{q-2}T_{p-1}\ar[ru]&&\Sigma^{q-1} T_{p-1}\ar@{~>}[ll]\ar[ru]&& \Sigma^{-1} T_{p-1} \ar@{~>}[ll]\ar[ru]&& T_{p-1}.\ar@{~>}[ll]
}
\end{gathered}
\end{align*}

Recall that by Corollary B, we have that
\begin{align*}
    K_0(\mathcal{C}_q(A_p))\cong K_0^{sp}(\mathcal{T})\Bigg/\Bigg\langle \xymatrix{\displaystyle\sum_{i=0}^{q} (-1)^i [\overline{T}_i]\,\,\,\Bigg\mid\,\,\, {\begin{matrix}M\in\text{Ind}\mathcal{T} \text{ with Auslander--Reiten } (q+3)\text{-angle }\\
    M\rightarrow \overline{T}_q\rightarrow\dots\rightarrow \overline{T}_0\rightarrow M\rightarrow \Sigma^{q+1}M\end{matrix}}
   }\Bigg\rangle.
\end{align*}
Using the Auslander--Reiten $(q+3)$-angles found and the fact that $q$ is odd, we obtain that in the quotient group on the right hand side, we have
\begin{align*}
    [T_1]=[T_{p-2}]=0 \,\,\, \text{ and } \,\,\, [T_{i-1}]=[T_{i+1}]\,\,\, \text{ for }\,\,\, 1\leq i\leq p-2.
\end{align*}
This implies that
\begin{itemize}
    \item if $p$ is even, then $[T_i]=0$ for all $0\leq i\leq p-1$,
    \item if $p$ is odd, then $0=[T_1]=\dots [T_{p-2}]$ and $0\neq[T_0]=\dots=[T_{p-1}]$.
\end{itemize}
Hence
\begin{align*}
    K_0(\mathcal{C}_q(A_p))\cong 
    \begin{cases}
    0 &\text{if $p$ is even,}\\
    \mathbb{Z} & \text{if $p$ is odd.}
    \end{cases}
\end{align*}
\end{proof}
\section{A higher angulated cluster category of type $A$}\label{section_eg2}
Let $p$ and $d$ be positive integers. We denote by $A^d_p$ the $(d-1)$-st higher Auslander $k$-algebra of linearly oriented $A_p$, see \cite[Section 3]{OT}. This is a $d$-representation finite algebra, in the sense that it has a $d$-cluster tilting module and $\text{gldim}(A^d_p)\leq d$, see \cite[Definition 2.19]{IOP}.
Let $\mmod A_p^d$ be the category of finitely generated $A_p^d$-modules and $\mathcal{D}^b(\mmod A_p^d)$ be its bounded derived category. We denote by $\mathbb{S}$ its Serre functor and by $\Sigma$ its suspension functor.

\begin{defn}[{\cite[construction 5.13]{OT}}]
For $\delta\geq d$, the \textit{$\delta$-Amiot cluster category of $A^d_p$} is defined to be
\begin{align*}
    \mathcal{C}^\delta({A_p^d})=\text{triangulated hull of }\mathcal{D}^b(\mmod A_p^d)/(\mathbb{S}_\delta),
\end{align*}
where $\mathbb{S}_\delta:=\mathbb{S}\Sigma^{-\delta}$.
\end{defn}

\begin{remark}
The category $\mathcal{C}^\delta({A_p^d})$ is a triangulated category containing the orbit category $\mathcal{D}^b(\mmod A_p^d)/(\mathbb{S}_\delta)$. We do not give a formal definition of triangulated hull here. Note that, by \cite[Theorem 5.14]{OT}, we have that if $\delta>d$, then $\mathcal{C}^\delta({A_p^d})$ is $\Hom$-finite and $\delta$-Calabi--Yau.
\end{remark}

\begin{remark}
Let $M$ be the unique $d$-cluster tilting object in $\mmod A^d_p$. Then
\begin{align*}
    \mathcal{U}:= \text{add}\{ \Sigma^{id}M \mid i\in\mathbb{Z} \}\subseteq \mathcal{D}^b(\mmod A^d_p)
\end{align*}
is a $d$-cluster tilting subcategory by \cite[Theorem 1.21]{I}.
\end{remark}

\begin{defn}[{\cite[Definition 5.22]{OT}}]
The \textit{$(d+2)$-angulated cluster category of $A^d_p$} is defined to be the orbit category
\begin{align*}
    \mathcal{O}({A^d_p})=\mathcal{U}/(\mathbb{S}_{2d})
\end{align*}
\end{defn}

\begin{remark}
Note that $\mathcal{O}({A^d_p})$ comes with an inclusion into $\mathcal{D}^b(\mmod A_p^d)/(\mathbb{S}_{2d})\subseteq \mathcal{C}^{2d}({A_p^d})$. Moreover, by \cite[Theorem 5.24]{OT}, we have that $\mathcal{O}({A^d_p})\subseteq\mathcal{C}^{2d}({A_p^d})$ is $d$-cluster tilting and $\mathcal{O}({A^d_p})$ is $(d+2)$-angulated.
\end{remark}

\begin{notation}
Let $Z$ be a \textit{cyclically ordered set} with $p+2d+1$ elements. We can think of $Z$ as marked points on a circle labeled $1$ to $p+2d+1$ in the anti-clockwise direction. Given three points $u,v,w$, we write $u<v<w$ if they appear in the order $u,v,w$ when going through the points in the anti-clockwise direction. Moreover, given two distinct points $u$ and $v$, we can consider the interval of points $[u,v]$ and in this ``$<$'' is a total order.

For a point $v$, we denote by $v^+$ its successor and by $v^-$ its predecessor in the anti-clockwise direction. We say that two points are neighbours if one is the successor of the other.
\end{notation}

\begin{lemma}[{\cite[proposition 6.10]{OT}}]
There is a bijection
\begin{align*}
   \text{Ind} \mathcal{O}({A^d_p})\longrightarrow \{ X=\{x_0,\,\dots,\, x_d\}\subset Z\mid X \text{ contains no neighbouring points}   \}.
\end{align*}
We will use it to identify the indecomposable objects of $\mathcal{O}(A^d_p)$ with the sets $X$. For $X=\{x_0,\,\dots,\, x_d\}\in \text{Ind} \mathcal{O}(A^d_p)$, we have that
\begin{align*}
    \Sigma^d X= \mathbb{S}_d X=\{ x_0^-,\,\dots,\, x_d^- \}.
\end{align*}
\end{lemma}

\begin{defn}
For $X,\, Y\in \text{Ind}\mathcal{O}({A^d_p})$, we say that $X$ \textit{intertwines} $Y$  if we can write $X=\{x_0,\,\dots,\, x_d\}$ and $Y=\{y_0,\,\dots,\, y_d\}$ such that
\begin{align*}
    x_0<y_0<x_1<\dots<y_{d-1}<x_d<y_d<x_0.
\end{align*}
Note that in this case also $Y$ intertwines $X$.
\end{defn}

\begin{lemma}[{\cite[proposition 6.1]{OT}}]
Given $X$ and $Y$ in Ind$\mathcal{O}({A^d_p})$, we have that $\Ext^d_{\mathcal{O}({A^d_p})}(X,Y)\neq 0$ if and only if $X$ intertwines $Y$. In this case, $\Ext^d_{\mathcal{O}({A^d_p})}(X,Y)$ is one-dimensional over $k$.
\end{lemma}

\begin{lemma}[{\cite[proposition 6.11]{OT}}]\label{lemma_angles_intert}
Let $X=\{x_0,\,\dots,\, x_d\}$ and $Y=\{y_0,\,\dots,\, y_d\}\in\text{Ind}\mathcal{O}({A^d_p})$ be such that
\begin{align*}
    x_0<y_0<x_1<\dots<y_{d-1}<x_d<y_d<x_0,
\end{align*}
so $X$ intertwines $Y$. Then there is a $(d+2)$-angle in $\mathcal{O}({A^d_p})$ of the form
\begin{align*}
    X\longrightarrow E_d\longrightarrow \dots\longrightarrow E_1\longrightarrow Y\longrightarrow \Sigma^d X \,\,\text{ with }\,\, E_r=\bigoplus_{\begin{smallmatrix}I\subseteq \{0,\dots, d\}\\ |I|\,=\,r\end{smallmatrix}} \{ x_i\mid i\in I \}\cup \{y_j\mid j\not\in I\},
\end{align*}
 where $\{ x_i\mid i\in I \}\cup \{y_j\mid j\not\in I\}$ is interpreted as zero if it contains neighbouring points.
\end{lemma}

\begin{lemma}
We have that $\mathcal{T}\subseteq \mathcal{O}({A^d_p})$ is Oppermann--Thomas cluster tilting if and only if $\rm{Ind }\mathcal{T}$ is a maximal set of non-intertwining elements in $\mathcal{O}({A^d_p})$ of cardinality
\begin{align*}
    \begin{pmatrix} p+d-1\\d \end{pmatrix}.
\end{align*}
\end{lemma}
\begin{proof}
By \cite[Theorem 6.4]{OT}, we have that $\mathcal{T}\subseteq \mathcal{O}({A^d_p})$ is Oppermann--Thomas cluster tilting if and only if it corresponds to a triangulation of $C(p+2d+1, 2d)$, in the notation of Oppermann and Thomas, see \cite[Page 2]{OT}. Moreover, by \cite[Theorems 2.3 and 2.4]{OT} such triangulations are precisely maximal sets of non-intertwining elements in $\mathcal{O}({A^d_p})$ of cardinality
\begin{align*}
    \begin{pmatrix} p+d-1\\d \end{pmatrix}.
\end{align*}
\end{proof}

\begin{remark}
Note that, by \cite[Theorem 5.25]{OT}, an object $T\in\mathcal{O}({A_p^d})$ is Oppermann--Thomas cluster tilting if and only if it is $2d$-cluster tilting when seen as an object in $\mathcal{C}^{2d}({A^d_p})$. 
\end{remark}
Hence, if we can find $\mathcal{T}=\text{add}(T)\subseteq \mathcal{O}({A_p^d})$ Oppermann--Thomas cluster tilting, we have
\begin{align*}
    \mathcal{T}\subseteq \mathcal{O}({A_p^d})\subseteq \mathcal{C}^{2d}({A^d_p}),
\end{align*}
where $\mathcal{C}^{2d}({A^d_p})$ is triangulated and $2d$-Calabi--Yau, $\mathcal{O}({A_p^d})$ is closed under $\Sigma^d$ and $d$-cluster tilting in $\mathcal{C}^{2d}({A^d_p})$ and $\mathcal{T}$ is $2d$-cluster tilting in $\mathcal{C}^{2d}({A^d_p})$.
That is, we are in the situation of Setup \ref{setup_T_S} with $\mathcal{S}=\mathcal{O}({A_p^d})$ and $\mathcal{C}=\mathcal{C}^{2d}({A^d_p})$.
We now choose specific values for $d$ and $p$ and, using our results, we find $K_0(\mathcal{C}^{2d}({A^d_p}))$ for these values. The following result will be widely used for the computations in our example.
\begin{proposition}[{\cite[Theorem 5.9]{JT}}]\label{prop_eg_PJ}
If $s_{d+1}\rightarrow \dots\rightarrow s_0\xrightarrow{\gamma}\Sigma^d s_{d+1}$ is a $(d+2)$-angle in $\mathcal{S}$, then
\begin{align*}
\sum_{i=0}^{d+1}(-1)^i\text{index}_{\mathcal{T}}(s_i)=\theta_{\mathcal{T}}([\im F_{\mathcal{T}}(\gamma)]).
\end{align*}
\end{proposition}

\begin{exmp}
Let $p=3$ and $d=2$, so that $|Z|=p+2d+1=8$. For simplicity, we write the indecomposable $\{ x_0, x_1, x_2 \}$ as $x_0x_1x_2$. We have
\begin{align*}
    \rm{Ind }\mathcal{O}({A^2_3})=\{ 135, 136, 137, 146, 147, 157, 246, 247, 248, 257, 258, 268, 357, 358, 368, 468 \}.
\end{align*}
Moreover, the object $T=135\oplus 136\oplus 137\oplus 146\oplus 147\oplus 157\in\mathcal{O}({A^2_3})$ is such that its indecomposable direct summands are a maximal set of non-intertwining elements in $\mathcal{O}({A^2_3})$ of cardinality  $\begin{psmallmatrix}3+2-1\\2\end{psmallmatrix}=\begin{psmallmatrix}4\\2\end{psmallmatrix}=6$. So $\mathcal{T}=\text{add}(T)\subset \mathcal{O}({A^2_3})$ is Oppermann--Thomas cluster tilting.

Using some $4$-angles in $\mathcal{O}({A^2_3})$ obtained as described in Lemma \ref{lemma_angles_intert} and \cite[lemma 5.6]{JT}, we find the index of the indecomposables in $\mathcal{O}({A^2_3})$ with respect to $\mathcal{T}$, see Table \ref{table:table_1}. Brackets $[- ]$ for classes in $K_0^{sp}(\mathcal{T})$ are omitted both in the table and in the rest of this example.
\begin{table}[h]
\begin{center}
\begin{tabular}{ c|c } 
 $s\in \mathcal{O}({A^2_3})$ & index$_\mathcal{T}(s)$\\
 \hline
 $135$ & $135$\\ 
 $136$ & $136$\\
 $137$ & $137$\\
 $146$ & $146$\\
 $147$ & $147$\\
 $157$ & $157$\\
 $246$ & $146-136+135$\\
 $247$ & $147+135-137$\\
 $248$ & $135$\\
 $257$ & $157-137+136$\\
 $258$ & $136$\\
 $268$ & $137$\\
 $357$ & $157-147+146$\\
 $358$ & $146$\\
 $368$ & $147$\\
 $468$ & $157$\\
\end{tabular}
\end{center}
\caption{The index of objects of $\mathcal{O}({A^2_3})$ with respect to $\mathcal{T}$.}
\label{table:table_1}
\end{table}

Consider the endomorphism algebra $\Gamma:=\End_{\mathcal{O}({A^2_3})}(T)$. The indecomposable projective $\Gamma$-modules are $P_x:=\Hom_{\mathcal{O}(A^2_3)}(T,x)$, for $x\in\mathcal{T}$ indecomposable. The simple top of $P_x$ is then denoted by $S_x$.
We compute $\theta_{\mathcal{T}}([S])$ for every simple $\Gamma$-module $S$.
In order to do this, we choose some morphisms $\gamma$ in $\mathcal{T}$, extend them to $4$-angles in $\mathcal{O}({A^2_3})$ using Lemma \ref{lemma_angles_intert}, and compute $\theta_{\mathcal{T}}([\im F_\mathcal{T}(\gamma)])$ using Proposition \ref{prop_eg_PJ} and Table \ref{table:table_1}, see Table \ref{table:table_2}. Then, since $\theta_{\mathcal{T}}$ is additive, we can compute $\theta_{\mathcal{T}}$ at the simple $\Gamma$-modules using Table \ref{table:table_2}, see Table \ref{table:table_3}.
\begin{table}
\begin{center}
\begin{tabular}{ c|c|c } 
 $s_3\rightarrow s_2\rightarrow s_1\rightarrow s_0\xrightarrow{\gamma} \Sigma^2 s_3$ & $[\im F_\mathcal{T}(\gamma)]\in K_0(\mmod \Gamma)$ & $\theta_{\mathcal{T}}([\im F_\mathcal{T}(\gamma)])$\\
 
 \hline
 
 $247\rightarrow 257\rightarrow 357\rightarrow 135\xrightarrow{\gamma} 136$ & $[S_{135}]$ & $136-146$ \\ 
 
 $257\rightarrow 357\rightarrow 135\rightarrow 136\xrightarrow{\gamma} 146$ & $[S_{136}]$ & $-135+137+146-147$\\
 
 $258\rightarrow 358\rightarrow 135\rightarrow 137\xrightarrow{\gamma} 147$ & $[S_{136}]+[S_{137}]$ & $-135-136+137+146$\\
 
 $258\rightarrow 268\rightarrow 468\rightarrow 146\xrightarrow{\gamma} 147$ & $[S_{136}]+[S_{146}]$ & $-136+137+146-157$\\
 
 $268\rightarrow 468\rightarrow 146\rightarrow 147\xrightarrow{\gamma} 157$ & $[S_{137}]+[S_{147}]$ & $-137-146+147+157$\\
 
 $268\rightarrow 0\rightarrow 0\rightarrow 157\xrightarrow{\gamma} 157$ & $[S_{137}]+[S_{147}]+[S_{157}]$ & $-137+157$\\
\end{tabular}
\end{center}
\caption{Evaluation of $\theta_{\mathcal{T}}$ at some useful values.}
\label{table:table_2}
\end{table}

\begin{table}
\begin{center}
\begin{tabular}{c|c } 
 $[S_x]\in K_0(\mmod \Gamma)$ & $\theta_{\mathcal{T}}([S_x])$\\
 
 \hline
 
 $[S_{135}]$ & $136-146$ \\ 
 
 $[S_{136}]$ & $-135+137+146-147$\\
 
 $[S_{137}]$ & $-135+137+146-147$\\
 
 $[S_{146}]$ & $135-136+147-157$\\
 
 $[S_{147}]$ & $136-137-146+157$\\
 
 $[S_{157}]$ & $146-147$\\
\end{tabular}
\end{center}
\caption{Evaluation of $\theta_{\mathcal{T}}$ at the simple $\Gamma$-modules $S_x$.}
\label{table:table_3}
\end{table}

Note that $\langle\theta_{\mathcal{T}}([S])\mid S \text{ is a simple } \Gamma-\text{module}\rangle$ generates $\im\theta_{\mathcal{T}}$ since $K_0(\mmod \Gamma)$ is generated by the classes of the simple $\Gamma$-modules. 
Hence, using Table \ref{table:table_3}, in $K_0^{sp}(\mathcal{T})/\im\theta_{\mathcal{T}}$ we have that
\begin{align*}
    136=146=147\,\,\, \text{ and }\,\,\, 137=135=157.
\end{align*}
By Remark \ref{remark_im_theta_gen}, Theorem C and Corollary E, we conclude that
\begin{align*}
    K_0 (\mathcal{C}^{4}(A^2_3))\cong K_0(\mathcal{O}({A_3^2}))\cong  K_0^{sp}(\mathcal{T})/\im\theta_{\mathcal{T}}\cong \mathbb{Z}\oplus\mathbb{Z}.
\end{align*}
\end{exmp}

\end{document}